\documentclass[reqno, 11pt]{amsart}
\usepackage{}
\usepackage{bbm}
\usepackage{url}
\usepackage{stmaryrd}
\usepackage{mathrsfs}
\usepackage{cases}
\usepackage{amsfonts}
\usepackage{amssymb}
\usepackage{amsmath,bm}
\usepackage{enumerate}
\usepackage{tikz}
\usepackage{extarrows}

\allowdisplaybreaks[4]
\newtheorem{theorem}{Theorem}[section]
\newtheorem{lemma}[theorem]{Lemma}
\newtheorem{proposition}[theorem]{Proposition}

\theoremstyle{definition}
\newtheorem{definition}[theorem]{Definition}

\numberwithin{equation}{section}

\let\al=\alpha
\let\b=\beta
\let\g=\gamma
\let\d=\delta

\let\la=\lambda

\let\s=\sigma

\let\G= \Gamma

\let\La=\Lambda

\let\wt=\widetilde

\let\va=\varphi

\let\fy=\infty
\def\bbR{\mathbb{R}}

\def\bbZ{\mathbb{Z}}
\def\bbJ{\mathbb{J}}

\def\calA{\mathcal {A}}
\def\calB{\mathcal {B}}

\def\calM{\mathcal {M}}

\def\calS{\mathcal {S}}

\def\scrF{\mathscr{F}}
\def\scrS{\mathscr{S}}
\def\scrL{\mathscr{L}}

\newcommand{\be}{\begin{equation*}}
\newcommand{\ee}{\end{equation*}}
\newcommand{\ben}{\begin{equation}}
\newcommand{\een}{\end{equation}}
\newcommand{\bn}{\begin{enumerate}}
\newcommand{\en}{\end{enumerate}}

\newcommand{\lan}{\langle}
\newcommand{\ran}{\rangle}

\def\mpqd{M^{p,q}(\rd)}
\def\mpq{M^{p,q}}
\def\mtmd{M^{p,q}(\rd)\longrightarrow M^{p,q}(\rd)}
\def\mtm{M^{p,q}(\mathbb{R})\longrightarrow M^{p,q}(\mathbb{R})}
\def\lpd{L^{p}(\rd)}

\def\lpqsd{L^{p,q}(\bbR^{2d})}

\def\fj{{\mathscr{F}_j}}
\def\fJ{{\mathscr{F}_{\mathbb{J}}}}
\def\fJmpq{{\mathscr{F}_{\mathbb{J}}M^{p,q}}}
\def\fJmpqd{{\mathscr{F}_{\mathbb{J}}M^{p,q}(\mathbb{R}^d)}}

\def\lqp{L^{q,p}}
\def\rr{{\mathbb R}}
\def\rd{{{\rr}^d}}
\def\rn{{\rr}^n}
\def\rdd{{{\rr}^{2d}}}
\def\bx{{\mathbf{x}}}
\def\bxi{{\bm{\xi}}}
\def\by{{\mathbf{y}}}
\def\zd{{{\mathbb{Z}}^d}}

\begin{document}
\title[Matrix dilation and Hausdorff operators on modulation spaces]
{Matrix dilation and Hausdorff operators on modulation spaces}
\author{WEICHAO GUO}
\address{School of Science, Jimei University, Xiamen, 361021, P.R.China}
\email{weichaoguomath@gmail.com}
\author{JIANGKUN LUO}
\address{School of Science, Jimei University, Xiamen, 361021, P.R.China}
\email{ljkmath@163.com}
\author{GUOPING ZHAO}
\address{School of Applied Mathematics, Xiamen University of Technology,
Xiamen, 361024, P.R.China} \email{guopingzhaomath@gmail.com}
\subjclass[2000]{42B15, 42B35.}
\keywords{dilation operator, Hausdorff operator, modulation spaces. }

\begin{abstract}
In this paper, we establish the asymptotic estimates for the norms of the matrix dilation operators on modulation spaces.
As an application, we study the boundedness on modulation spaces of Hausdorff operators.
The definition of Hausdorff operators are also revisited for fitting our study.
\end{abstract}

\maketitle


\section{INTRODUCTION}

Different function spaces are used in different ways to measure the
content of a distribution.
In order to measure both the time and frequency contents,
modulation space was first introduced by H. Feichtinger \cite{Feichtinger1983TRUoV} in 1983.
Now, the modulation space has been studied extensively.
It has turned out to be the appropriate function spaces in the field of time-frequency analysis,
see Gr\"{o}chenig's book \cite{GrochenigBook2013}.
More precisely, modulation spaces
are defined by measuring the decay and integrability of the STFT as following:
\be
M^{p,q}(\rd)=\{f\in \calS'(\rd): V_gf\in L^{p,q}(\rdd) \}
\ee
endowed with the obvious (quasi-)norm, where $L^{p,q}(\rdd)$ are mixed-norm Lebesgue spaces,
STFT denotes the short-time Fourier transform,
more details can be found in Section 2. By $\calM^{p,q}(\rd)$ we denote the $\calS(\rd)$ closure in $M^{p,q}(\rd)$.

In this paper, we focus on the estimates for the norm of the dilation operators on modulation spaces.
As a fundamental property of function spaces, this topic has been well studied.
Before we embark on this topic, let us first clarify some definitions and notations.

For a non-degenerate matrix $A\in GL(d,\rr)$, the $A$-dilation operator $D_A$ is defined by
\be
D_A: f(x):\longrightarrow f(Ax),\ \ \ x\in \rd.
\ee
When $A=\la I$ with $\la>0$ and the $d\times d$ identity matrix $I$, we use the denotation
\be
D_{\la}:=D_{\la I}.
\ee

For $1\leq p,q\leq \infty$, we define
$\mu_1(p,q)=-1/p\vee (1/q-1)\vee (-2/p+1/q)$,
$\mu_2(p,q)=-1/p\wedge (1/q-1)\wedge (-2/p+1/q)$, that is,
\begin{equation*}
  \mu_1(p,q)=
  \begin{cases}
  -1/p,  &\text{if } (1/p,1/q)\in \calA_1:\, 1/q\leq (1-1/p)\wedge 1/p,\\
  1/q-1, &\text{if } (1/p,1/q)\in \calA_2:\, 1/p\geqslant (1-1/q)\vee 1/2,\\
  -2/p+1/q, &\text{if } (1/p,1/q)\in \calA_3:\, 1/p \leq 1/q \wedge 1/2;
  \end{cases}
\end{equation*}
\begin{equation*}
  \mu_2(p,q)=
  \begin{cases}
  -1/p,  &\text{if } (1/p,1/q)\in \calB_1:\, 1/q\geqslant (1-1/p)\vee 1/p,\\
  1/q-1, &\text{if } (1/p,1/q)\in \calB_2:\, 1/p\leq (1-1/q)\wedge 1/2, \\
  -2/p+1/q, &\text{if } (1/p,1/q)\in \calB_3:\, 1/p\geqslant 1/q\vee 1/2.
  \end{cases}
\end{equation*}
See the following figures for the sets $\calA_i$ and $\calB_i$, $i=1,2,3$.

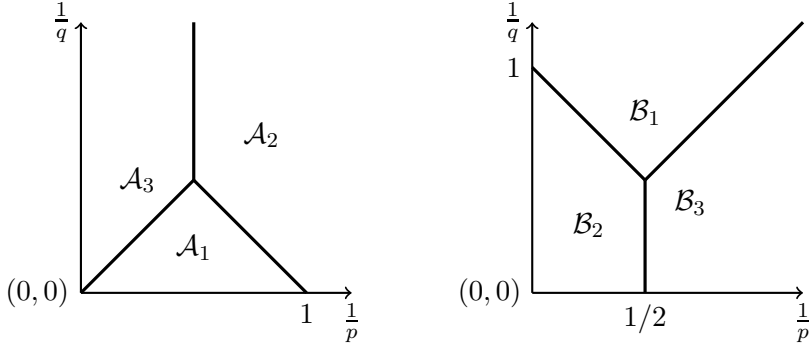
\begin{tikzpicture}[scale=3]
\coordinate (Origin) at (0,0);
\coordinate(horizontal_axis) at (1.2,0);
\coordinate(vertical_axis) at (0,1.2);
\coordinate (m10) at (1,0);
\coordinate (m0505) at (0.5,0.5);
\coordinate (m11) at (1.2,1.2);
\coordinate (m051) at (0.5,1.2);
\draw[thick,->] (Origin) -- (horizontal_axis) node[below]{$\frac{1}{p}$};
\draw[thick,->] (Origin) -- (vertical_axis) node[left]{$\frac{1}{q}$};
\draw [very thick] (m0505)-- (m10) node [below] {$1$};
\draw [very thick] (m0505)--(Origin) node [left] {$(0,0)$};
\draw [very thick] (m0505)-- (m051);
\draw(0.8,0.7)node{$\calA_2$};
\draw(0.5,0.2)node{$\calA_1$};
\draw(0.25,0.5)node{$\calA_3$};
\node at (0.6,-0.25){};
\coordinate (2Origin) at (2,0);
\coordinate(2horizontal_axis) at (3.2,0);
\coordinate(2vertical_axis) at (2,1.2);
\coordinate (2m050) at (2.5,0);
\coordinate (2m0505) at (2.5,0.5);
\coordinate (2m11) at (3.2,1.2);
\coordinate (2m01) at (2,1);
\draw[thick,->] (2Origin) -- (2horizontal_axis) node[below]{$\frac{1}{p}$};
\draw[thick,->] (2Origin) -- (2vertical_axis) node[left]{$\frac{1}{q}$};
\draw [very thick] (2m0505)-- (2m01) node [left] {$1$};
\draw [very thick] (2m11)--(2m0505);
\draw (2Origin) node [left] {$(0,0)$};
\draw [very thick] (2m0505)-- (2m050) node [below] {$1/2$};
\draw(2.7,0.4)node{$\calB_3$};
\draw(2.25,0.3)node{$\calB_2$};
\draw(2.5,0.8)node{$\calB_1$};
\node at (2.6,-0.25){};
\end{tikzpicture}

In order to study the embedding relations between modulation and Besov spaces,
the action of $D_{\la}$ on modulation spaces has been studied carefully, by Sugimoto--Tomita \cite{SugimotoTomita2007JoFA}.
We rewrite the corresponding result in \cite{SugimotoTomita2007JoFA} as follows.
See also \cite{Cordero2009} for the dilation property on Wiener amalgam space.

\hspace{-12pt}\textbf{Theorem A} (cf. \cite[Theorem 3.1]{SugimotoTomita2007JoFA})\quad
Let $1\leq p,q\leq \infty$.
Then the following statements are true:
\bn
\item
There exists a constant $C>0$ such that
\be
\|D_{\la}\|_{\mtmd}\leq C \la^{d\mu_1(p,q)},\ \ \ \ \text{for all}\  \la\in [1,\fy).
\ee
Conversely, if there exists a constant $C>0$ such that
\be
\|D_{\la}\|_{\mtmd}\leq C\la^{\al},\ \ \ \ \text{for all}\  \la\in [1,\fy),
\ee
then $\al\geq d\mu_1(p,q)$.
\item
There exists a constant $C>0$ such that
\be
\|D_{\la}\|_{\mtmd}\leq C \la^{d\mu_2(p,q)},\ \ \ \ \text{for all}\  \la\in (0,1].
\ee
Conversely, if there exists a constant $C>0$ such that
\be
\|D_{\la}\|_{\mtmd}\leq C\la^{\b},\ \ \ \ \text{for all}\  \la\in (0,1],
\ee
then $\b\leq d\mu_2(p,q)$.
\en

Then, the general case, that is, the matrix dilation operators, was studied by Cordero-Nicola \cite{CorderoNicola2007JoFA},
in which they obtain the following results.

\hspace{-12pt}\textbf{Theorem B} (cf. \cite[Theorem 3.4]{CorderoNicola2007JoFA})\quad
Let $1\leq p,q\leq \infty$.
There exists a constant $C$ such that, for every symmetric matrix $A\in GL(d,\rr)$ with eigenvalues $\{\la_j\}_{j=1}^d$,
we have

\be
\|D_{A}\|_{\mtmd}\leq C \prod_{j=1}^d (1\vee\la_j)^{\mu_1(p,q)}(1\wedge\la_j)^{\mu_2(p,q)}.
\ee
Conversely, if there exists a constant $C>0$ such that
\be
\|D_{A}\|_{\mtmd}\leq C\prod_{j=1}^d (1\vee\la_j)^{\al_j}(1\wedge\la_j)^{\b_j}
\ee
with $A=diag(\la_1,\cdots,\la_d)$,
then $\al_j\geq \mu_1(p,q)$ and $\b_j\leq \mu_1(p,q)$.

In \cite{CorderoNicola2007JoFA}, the authors also consider the more general case for the matrix $A\in GL(d,\rr)$ without assuming symmetry.
One can see \cite[Proposition 3.1]{CorderoNicola2007JoFA}), however, this result is not sharp compared with Theorem B.
In fact, even the result in Theorem B leaves some room for improvement.
A simple reason is that the estimates in Theorems A and B only give the optimal exponent in the framework of power functions.
There exist many other possible functions that cannot be determined by the estimations in Theorems A and B.
For instance, let $h(\la)=\la^a\ln^b(e+\la)$ with $a>0$, $b<0$. One can check that $h(\la)\leq C\la^a$ for $\la$ and
\be
h(\la)\leq C \la^{\al} \text{ for } \la\geq1 \text{ implies } \al\geq a.
\ee
To fill this gap, we will give a direct asymptotic estimate for the norm
of matrix dilation operators with general matrix $A\in GL(d,\rr)$.

For simplicity, we denote
\be
\G_{p,q}(\la)=\max\{\la^{-1/p}, \la^{1/q-1}, \la^{-2/p+1/q}\},\ \ \ \la>0.
\ee
Then,
\be
\G_{p,q}(\la)=\la^{\mu_1(p,q)}\ \text{for}\ \la\in [1,\fy),\ \ \
\G_{p,q}(\la)=\la^{\mu_2(p,q)}\ \text{for}\ \la\in (0,1].
\ee
Our first main theorem gives the asymptotic estimate of the matrix dilation $D_A$.

\begin{theorem}\label{thm-DL}
  Let $1\leq p, q \leq \fy$.
  There exist two constants $C_1$ and $C_2$ such that for all non-degenerate matrix $A\in GL(d,\rr)$ with
  singular values $\{\la_j\}_{j=1}^d$,
we have the estimate
  \ben
C_1
\prod_{j=1}^d\G_{p,q}(\la_j)
\leq
\|D_{A}\|_{\mtmd}\leq C_2 \prod_{j=1}^d\G_{p,q}(\la_j).
  \een
\end{theorem}

For the upper bound estimate,
our theorem is an extension of Theorems A and B to the more general matrix dilation operators.
For the lower bound estimate, our theorem is an essential improvement of Theorems A and B, even for the
scalar matrix dilation operator $D_{\la}$.

After establishing Theorem \ref{thm-DL},
a natural idea is to apply it to the study of Hausdorff operators.

For a suitable function $\Phi$, the corresponding Hausdorff operator $H_{\Phi}$ can be formally defined by
\begin{equation}
  H_{\Phi}f(x):=\int_{\rn}\Phi(y)D_{1/|y|}f(x)dy.
\end{equation}

The study of Hausdorff operators was originated from some classical summation methods.
Today, it has attracted more and more attention of many researchers.
One can see
\cite{Chen2013} and \cite{Liflyand2011} for a survey with some historical background and recent developments
regarding Hausdorff operators.

When $\Phi$ is taken suitably, Hausdorff operator contains some important operators in the field of harmonic analysis.
For instance, the Hardy operator, adjoint Hardy operator (see \cite{ChenFanLi2012CAMSB, ChenFanZhang2012AMSES, FanLin2014AB}),
and the Ces\`{a}ro operator \cite{Miyachi2004JFAA, Siskakis1990PAMS} in one dimension.

For a matrix-valued function $A(y)\in GL(d,\rr)$ for $y\in \rn$,
a more general Hausdorff operator is defined by
\be
H_{\Phi,A}f(x): =\int_{\rn}\Phi(y)(D_{A(y)}f)(x)dy.
\ee
This general Hausdorff operator was introduced by Brown-M\'{o}ricz\cite{BrownMoricz2002JMAA}
and Lerner--Liflyand \cite{LernerLiflyand2007JAMS}.
One can check that $H_{\Phi}=H_{\Phi,A}$ with $A(y)=diag(1/|y|,\cdots,1/|y|)$.
In this sense, $H_{\Phi}$ is a special case of the general Hausdorff operaotr $H_{\Phi,A}$.

By the definition of Hausdorff operator, one can see that the dilation property of the function space has closed relations with
the boundedness on the corresponding function space of Hausdorff operator. With the help of Theorem A, we have
studied the action of $H_{\Phi}$ on modulation spaces in \cite{ZhaoFanGuo2018AFA}.
In this paper, our second main goal is to study the boundedness on modulation spaces of $H_{\Phi,A}$.
In particular, we establish the sharp conditions for the boundedness of $H_{\Phi,A}$ on $M^{p,q}$.
Our second main theorem is as follows. This theorem generalizes the main results in \cite{ZhaoFanGuo2018AFA}.
\begin{theorem}\label{thm-bdh}
Let $1\leq p, q \leq \fy$.
Assume that $A(y)\in GL(d,\rr)$ is non-degenerate matrix for all $y\in\rn$, with
  singular values $\{\la_j(y)\}_{j=1}^d$.
If the following condition holds
  \ben\label{thm-bdh-cd}
    \int_{\rn}|\Phi(y)|\cdot
       \prod_{j=1}^d \G_{p,q}(\lambda_j(y)) dy
    <
    \infty,
  \een
we have the boundedness of
\ben\label{thm-bdh-bd}
H_{\Phi, A}: \calM^{p,q}(\rd) \longrightarrow \mpqd.
\een
In addition, if $\Phi\geq0$, $(1/p-1/2)(1/q-1/p)\geqslant 0$, and $A(y)=\La(y)=\text{diag}\{\lambda_1(y),\cdots,\lambda_d(y) \}$,
then the converse direction is valid, that is, we have the equivalent relation
$\eqref{thm-bdh-bd}\Longleftrightarrow \eqref{thm-bdh-cd}$.
Here the boundedness \eqref{thm-bdh-bd} means $\|H_{\Phi, A}\|_{\mpqd}\leq C\|f\|_{\mpqd}$ for all $f\in\scrS(\rd)$. 
Then the boundedness can be extended to
$\calM^{p,q}(\rd)$.
\end{theorem}

Our paper is organized as follows.
In Section 2, we collect some notations and basic properties of modulation spaces.
Section 3 is devoted to the estimates of matrix dilation operators.
The proof of Theorem \ref{thm-DL} will be given in this section.
In Section 4, we first revisit the definition of $H_{\Phi,A}$, giving a reasonable condition on $\Phi$ to
ensure that the action on modulation of $H_{\Phi,A}$ can be well defined.
Then, by a new embedding relation between partial Fourier modulation and mixed-norm spaces,
and a lower estimate of $H_{\Phi,A}f$ in the corresponding mixed-norm spaces, we give the proof of Theorem \ref{thm-bdh}.

Throughout this paper, we will adopt the following notations.
We use $X\lesssim Y$ to denote the statement that $X\leqslant CY$, with a positive constant $C$ that may depend on $p, q, d$,
but it might be different from line to line.
The notation $X\sim Y$ means the statement $X\lesssim Y\lesssim X$.
For a multi-index $k=(k_{1},k_{2},...,k_{n})\in \mathbb{Z}^{n}$,
we denote $|k|_{\infty }:=\max\limits_{i=1,2,...,n}|k_{i}|$ and $\langle k\rangle:=(1+|k|^{2})^{{1}/{2}}.$
We use $\scrL$ to denote a large number might change from line to line.

\section{PRELIMINARIES}
For any fixed $x, \xi\in \rd$, the translation operator $T_x$ and modulation operator $M_{\xi}$ are defined, respectively, by
\be
T_xf(t)=f(t-x),\ \ \ \ M_{\xi}f(t)=e^{2\pi it\cdot\xi}f(t).
\ee

The short-time Fourier transform (STFT) of a function $f$ with respect to a window $g$ is defined by
\be
V_gf(x,\xi):=\int_{\rd}f(t)\overline{g(t-x)}e^{-2\pi it\cdot \xi}dt,\ \ \  f,g\in L^2(\rd).
\ee
Its extension to $\calS'\times \calS$ can be denoted by
\be
V_gf(x,\xi)=\langle f, M_{\xi}T_xg\rangle,
\ee
in which the STFT $V_gf$ is a bilinear map from $\calS'(\rd)\times \calS(\rd)$ into $\calS'(\rdd)$.
The so-called fundamental identity of time-frequency analysis is as follows:
\ben\label{itf}
V_gf(x,\xi)=e^{-2\pi ix\cdot \xi}V_{\hat{g}}\hat{f}(\xi,-x),\ \ \ (x,\xi)\in \rdd.
\een

The weighted mixed-norm spaces used to measure the STFT are defined as follows.

\begin{definition}[Mixed-norm spaces.]
Let $p,q\in (0,\fy]$, and $d_1, d_2\in\bbZ^+$. 
Then the mixed-norm space $L^{p,q}(\bbR^{d_1}\times\bbR^{d_2})$
consists of all Lebesgue measurable functions on $\bbR^{d_1}\times\bbR^{d_2}$ such that the (quasi-)norm
\be
\|F\|_{L^{p,q}(\bbR^{d_1}\times\bbR^{d_2})}
=
\left(\int_{\bbR^{d_2}}\left(\int_{\bbR^{d_1}}|F(x,\xi)|^pdx\right)^{q/p}d\xi\right)^{1/q}
\ee
is finite, with usual modification when $p=\fy$ or $q=\fy$.
For the convenience of writing, we denote $L^{p,q}(\rdd):=L^{p,q}(\bbR^{d}\times\bbR^{d})$.
\end{definition}
\begin{lemma}[Young's inequality of mixed-norm spaces, \cite{GrochenigBook2013}]\label{lm-Young-mixed-norm}
	If $F\in L^{1,1}(\bbR^{2d})$, and $G\in \lpqsd$, then
	\[
	  \|F \ast G\|_{\lpqsd}\lesssim \|F\|_{L^{1,1}(\bbR^{2d})} \|G\|_{\lpqsd}.
	\]
\end{lemma}

Now, we introduce the definition of modulation space.
\begin{definition}\label{df-M}
Let $0<p,q\leq \infty$.
Given a non-zero window function $\phi\in \calS(\rd)$, the (weighted) modulation space $M^{p,q}_m(\rd)$ consists
of all $f\in \calS'(\rd)$ such that the norm
\be
\begin{split}
\|f\|_{M^{p,q}(\rd)}&:=\|V_{\phi}f\|_{L^{p,q}(\rdd)}
=\left(\int_{\rd}\left(\int_{\rd}|V_{\phi}f(x,\xi)|^{p} dx\right)^{{q}/{p}}d\xi\right)^{{1}/{q}}
\end{split}
\ee
is finite, with usual modification when $p=\infty$ or $q=\infty$.
\end{definition}
Note that the above definition of $M^{p,q}$ is independent of the choice of window function $\phi$
in the sense of equivalent norms.

Applying the frequency-uniform localization techniques, one can give an alternative definition of modulation spaces (see \cite{Triebel1983ZFA} for details).

We denote by $Q_{k}$ the unit cube with the center at $k$. Then the family $\{Q_{k}\}_{k\in\mathbb{Z}^{d}}$
constitutes a decomposition of $\mathbb{R}^{d}$.
Let $\rho \in \mathscr {S}(\mathbb{R}^{d}),$
$\rho: \rd \rightarrow [0,1]$ be a smooth function satisfying that $\rho(\xi)=1$ for
$|\xi|_{\infty}\leq {1}/{2}$ and $\rho(\xi)=0$ for $|\xi|\geq 3/4$. Let
\begin{equation}
\rho_{k}(\xi)=\rho(\xi-k),  k\in \zd
\end{equation}
be a translation of \ $\rho$.
Since $\rho_{k}(\xi)=1$ in $Q_{k}$, we have that $\sum_{k\in\zd}\rho_{k}(\xi)\geq1$
for all $\xi\in\rd$. Denote
\begin{equation}
\sigma_{k}(\xi)=\rho_{k}(\xi)\bigg(\sum_{l\in\zd}\rho_{l}(\xi)\bigg)^{-1},  ~~~~ k\in\zd.
\end{equation}
It is easy to know that $\{\sigma_{k}\}_{k\in\zd}$
constitutes a smooth decomposition of $\rd$, and $%
\sigma_{k}(\xi)=\sigma_0(\xi-k)=\sigma(\xi-k)$. The frequency-uniform decomposition
operators can be defined by
\begin{equation}
\Box_{k}:= \mathscr{F}^{-1}\sigma_{k}\mathscr{F}
\end{equation}
for $k\in \zd$.
Now, we give the (discrete) definition of modulation space $\mpqd$.

\begin{definition}\label{df-Mc}
Let $0<p,q\leq \infty$. The modulation space $\mpqd$ consists of all $f\in \mathscr{S}'(\rd)$ such that the (quasi-)norm
\begin{equation}
\|f\|_{\mpqd}:=\bigg( \sum_{k\in \mathbb{Z}^d}\langle k\rangle ^{sq}\|\Box_k f\|_{L^p}^{q}\bigg)^{1/q}
\end{equation}
is finite, with usual modification when $q=\infty$.
\end{definition}

We remark that the above definition is independent of the choice of decomposition function (see \cite{Wang2011Book}).
So we can use appropriate $\sigma$ according to the problem we deal with.
We also recall that the definitions \ref{df-M} and \ref{df-Mc} are equivalent.

The following two lemmas establish two equivalent relations between Lebesgue and modulation spaces from two local perspectives.

\begin{lemma}(See \cite[Lemma 3.2]{Cordero2009})\label{lm-lp-1}
	Let $1\leq p,q\leq\infty$ and $R>0$.
	Suppose $\text{supp}\scrF f\subset B(0,R)$, then
  \be
  \|f\|_{M^{p,q}(\rd)}\sim_R \|f\|_{L^p(\rd)}.
  \ee
\end{lemma}

\begin{lemma}(See \cite[Lemma 3.2]{Cordero2009})\label{lm-lp-2}
	Let $1\leq p,q\leq\infty$ and $R>0$.
	Suppose $\text{supp} f\subset B(0,R)$, then
  \be
  \|f\|_{M^{p,q}(\rd)}\sim_R \|f\|_{\scrF L^q(\rd)}.
  \ee
\end{lemma}

Next, we give two basic results of the calculation of modulation space, which will be used in the proof of our main theorems.

\begin{lemma}\label{lm-sp}
Let $1\leq p,q\leq \infty$.
  Assume that $f_j\in \mpq$, $j=1,2,\cdots,d$. Then, the function $F=\otimes_{j=1}^df_j$ belongs to $\mpqd$, and we have
  \be
  \|F\|_{\mpqd}=\prod_{j=1}^d\|f_j\|_{\mpq}.
  \ee
  More precisely, the following equality is valid:
  \be
  \|V_{\otimes_{j=1}^d \phi_j}(\otimes_{j=1}^d f_j)\|_{L^{p,q}(\rdd)}
  =
  \prod_{j=1}^d \|V_{\phi_j}f_j\|_{L^{p,q}(\rr^2)},
  \ee
  where $\phi_j\in\calS\setminus\{0\}$, $j=1,2,\cdots,d$.
\end{lemma}
\begin{proof}
  For $\phi_j\in \calS(\rr)$, $j=1,2,\cdots,d$,
  a direct calculation yields that
  \be
  \begin{split}
    &V_{\otimes_{j=1}^d \phi_j}(\otimes_{j=1}^d f_j)(x,\xi)
    \\
    = &
    \int_{\rd}\prod_{j=1}^d f_j(y_j)\prod_{j=1}^d\overline{\phi_j(y_j-x_j)}e^{-2\pi i\sum_{j=1}^dy_j\cdot\xi_j}dy
    \\
    = &
    \prod_{j=1}^d\int_{\rr} f_j(y_j)\overline{\phi_j(y_j-x_j)}e^{-2\pi iy_j\cdot\xi_j}dy_j
    \\
    =&
    \prod_{j=1}^d V_{\phi_j}f_j(x_j,\xi_j).
  \end{split}
  \ee
  The desired conclusion follows by taking the $L^{p,q}(\rdd)$ norm on both sides of the above equality.
\end{proof}

\begin{lemma}\label{lm-orthogonal}
	Let $1\leq p,q\leq \infty$ and let $P$ be an orthogonal matrix.
	Then $\|D_P f\|_{\mpqd} = \|f\|_{\mpqd}$ for all $f\in \mpqd$.
\end{lemma}
\begin{proof}
	Let $\varphi\in \mathscr{S}$ be a nonzero radial function.
	By a direct calculation, we have
	\be
	\begin{split}
		V_{\varphi}D_Pf(x,\xi)
		=&
		\int_{\bbR^d}D_Pf(y) \overline{\varphi(y-x)} e^{-2\pi i y\cdot\xi}dy
		\\
		=&
		\int_{\bbR^d}f(y) \overline{\varphi(P^Ty-x)} e^{-2\pi i P^Ty\cdot\xi}dy
		\\
		=&
		\int_{\bbR^d}f(y) \overline{\varphi(y-Px)} e^{-2\pi i y\cdot P\xi}dy
		\\
		=&
		V_{\varphi}f(Px,P\xi).
	\end{split}
	\ee
	Combining this with the definition of modulation spaces, we obtain
	\be
	\begin{split}
		\|D_Pf\|_{\mpqd}
		=
		\|V_{\varphi}D_Pf\|_{L^{p,q}(\rdd)}
		=
		\|V_{\varphi}f(P\cdot)\|_{L^{p,q}(\rdd)}
		=
		\|V_{\varphi}f\|_{L^{p,q}(\rdd)}
		=
		\|f\|_{\mpqd}.
	\end{split}
	\ee
\end{proof}

\section{Matrix dilation operators}
In this section, we explore the asymptotic estimates of the matrix dilation operators $D_{A}$ on modulation spaces.
Our strategy is to establish the estimates gradually from the particular case to the general one.
Let us start with the scalar matrix case, that is, the dilation operator $D_{\la}$.

\subsection{Scalar matrix dilation}
In order to get the lower bound estimates of $\|D_{\la}\|_{M^{p,q}\longrightarrow M^{p,q}}$, we first give some useful estimates for $\|D_{\la}f\|_{M^{p,q}}$ with certain functions $f$.

\begin{lemma}\label{lm-lb1}
  Let $1\leq p,q\leq \fy$, and
  that $g_1$ be a nonzero smooth function with $\text{supp}\widehat{g_1}\subset B(0,1)$, then
  \be
  \|D_{\la}g_1\|_{M^{p,q}(\rd)}\sim \|D_{\la}g_1\|_{L^p(\rd)}=\la^{-d/p}\|g_1\|_{L^p(\rd)},
  \ \ \ \ \la\in (0,1].
  \ee
\end{lemma}
\begin{proof}
  Observe that
  $\text{supp}\widehat{D_{\la}g_1}\subset B(0,1)$ for all $\la\in (0,1]$.
  The desired conclusion follows by Lemma \ref{lm-lp-1}.
\end{proof}

\begin{lemma}\label{lm-lb2}
Let $1\leq p,q\leq \fy$, and
that $g_2$ be a nonzero smooth function supported in $B(0,1)$, then
  \be
  \|D_{\la}g_2\|_{M^{p,q}(\rd)}\sim \|D_{\la}g_2\|_{\scrF L^q(\rd)}=\la^{d(1/q-1)}\|g_2\|_{\scrF L^q(\rd)},
  \ \ \ \ \la\in [1,\fy).
  \ee
\end{lemma}
\begin{proof}
  Observe that
  $\text{supp}D_{\la}g_2\subset B(0,1)$ for all $\la\in [1,\fy)$.
  The desired conclusion follows by Lemma \ref{lm-lp-2}.
\end{proof}

\begin{lemma}\label{lm-lb3}
Let $1\leq p,q\leq \fy$, and
that $h$ be a smooth function with $\text{supp}\widehat{h}\subset [-1/4,1/4]^d$ and $\|h\|_{L^p}=1$.
Let $f_k(x):=M_kf(x)=e^{2\pi ix\cdot k}h(x)$.
For $L>0$ we set
\be
F_{\la,L}(x)=\sum_{|k|_{\fy}\leq \la^{-1}}T_{Lk}f_k(x)
=
\sum_{|k|_{\fy}\leq \la^{-1}}f_k(x-Lk).
\ee
We have
\be
\lim_{L\rightarrow \fy}\frac{\|D_{\la}F_{\la,L}\|_{M^{p,q}}}{\|F_{\la,L}\|_{M^{p,q}}}\sim \la^{d(-2/p+1/q)},\ \ \ \la\in (0,1].
\ee
\begin{proof}
  By the definition of $F_{\la,L}$ and $\Box_k$,
  we have
  \be
  \Box_kF_{\la,L}=\Box_k(T_{Lk}f_k)=T_{Lk}f_k.
  \ee
  Then, a direct calculation yields that
  \ben\label{lm-lb3-1}
  \begin{split}
  \|F_{\la,L}\|_{\mpqd}
  =&
  \bigg(\sum_{|k|_{\fy}\leq \la^{-1}}\|\Box_kF_{\la,L}\|_{L^p}^q\bigg)^{1/q}
  =
  \bigg(\sum_{|k|_{\fy}\leq \la^{-1}}\|T_{Lk}f_k\|_{L^p}^q\bigg)^{1/q}
  \\=&
  \bigg(\sum_{|k|_{\fy}\leq \la^{-1}}\|f_k\|_{L^p}^q\bigg)^{1/q}
  =
  \bigg(\sum_{|k|_{\fy}\leq \la^{-1}}\|h\|_{L^p}^q\bigg)^{1/q}\sim \la^{-d/q}.
  \end{split}
  \een
  On the other hand, observe that
  \be
  \text{supp}\scrF D_{\la}(F_{\la,L})\subset [-2,2]^d.
  \ee
  By Lemma \ref{lm-lp-1} we have
  \be
  \|D_{\la}(F_{\la,L})\|_{\mpqd}
  \sim
  \|D_{\la}(F_{\la,L})\|_{\lpd}
  =
  \la^{-d/p}\|F_{\la,L}\|_{\lpd}.
  \ee
  Form this and the almost orthogonality of $\{T_{Lk}f_k\}_{|k|_{\fy}\leq \la^{-1}}$, we obtain that
  \ben\label{lm-lb3-2}
  \begin{split}
  \lim_{L\rightarrow \fy}\|D_{\la}(F_{\la,L})\|_{\mpqd}
  = &
  \lim_{L\rightarrow \fy}\la^{-d/p}\|F_{\la,L}\|_{\lpd}
  \\
  = &
  \la^{-d/p}\bigg( \lim_{L\rightarrow \fy}\sum_{|k|_{\fy}\leq \la^{-1}}\|T_{Lk}f_k(x)\|_{\lpd}^p\bigg)^{1/p}
  \\
  = &
  \la^{-d/p}\bigg(\sum_{|k|_{\fy}\leq \la^{-1}}\|h\|_{\lpd}^p\bigg)^{1/p}
  \sim \la^{-2d/p}.
  \end{split}
  \een
  The final conclusion follows by \eqref{lm-lb3-1} and \eqref{lm-lb3-2}.
\end{proof}
\end{lemma}

With the above estimates, we establish the asymptotic estimates of $\|D_{\la}\|_{M^{p,q}\longrightarrow M^{p,q}}$ as follows.

\begin{proposition}\label{pp-scalar}
Let $1\leq p,q\leq \infty$, $\la>0$.
Then, there exist two positive constants $C_1$ and $C_2$ such that
\be
C_1\G_{p,q}(\la)^d
\leq
\|D_{\la}\|_{\mtmd}\leq C_2 \G_{p,q}(\la)^d.
\ee
\end{proposition}
\begin{proof}
The upper boundness follows by Theorem A. We turn to the estimates of lower bound.

Let $g_1$ be a nonzero smooth function with $\text{supp}\widehat{g_1}\subset B(0,1)$, then
by Lemma \ref{lm-lb1}, we have
\be
\|D_{\la}\|_{\mtmd}\geq \frac{\|D_{\la}g_1\|_{\mpqd}}{\|g_1\|_{\mpqd}}\sim \la^{-d/p},\ \ \ \la\in (0,1],
\ee
and
\be
\|D_{\la}\|_{\mtmd}\geq \frac{\|D_{\la}\circ D_{1/\la}g_1\|_{\mpqd}}{\|D_{1/\la}g_1\|_{\mpqd}}
=
\frac{\|g_1\|_{\mpqd}}{\|D_{1/\la}g_1\|_{\mpqd}}
\sim \la^{-d/p},\ \ \la\in [1,\fy].
\ee
Let $g_2$ be a nonzero smooth function supported in $B(0,1)$, then by Lemma \ref{lm-lb2}, we have
\be
\|D_{\la}\|_{\mtmd}\geq \frac{\|D_{\la}g_2\|_{\mpqd}}{\|g_2\|_{\mpqd}}\sim \la^{d(1/q-1)},\ \ \ \ \la\in [1,\fy),
\ee
and
\be
\|D_{\la}\|_{\mtmd}\geq \frac{\|D_{\la}\circ D_{1/\la}g_2\|_{\mpqd}}{\|D_{1/\la}g_2\|_{\mpqd}}
=
\frac{\|g_2\|_{\mpqd}}{\|D_{1/\la}g_2\|_{\mpqd}}
\sim \la^{d(1/q-1)},\ \  \la\in (0,1].
\ee
Let $F_{\la, L}$ be stated as in Lemma \ref{lm-lb3} and by Lemma \ref{lm-lb3}, we have
\be
\begin{split}
\|D_{\la}\|_{\mtmd}
\geq &
\lim_{L\rightarrow \fy}\frac{\|D_{\la}F_{\la,L}\|_{\mpqd}}{\|F_{\la,L}\|_{\mpqd}}\sim \la^{d(-2/p+1/q)},\ \ \ \la\in (0,1],
\end{split}
\ee
and
\be
\begin{split}
\|D_{\la}\|_{\mtmd}
\geq &
\lim_{L\rightarrow \fy}\frac{\|D_{\la}\circ D_{1/\la}F_{1/\la,L}\|_{\mpqd}}{\|D_{1/\la}F_{1/\la,L}\|_{\mpqd}}
\\
= &
\frac{\|F_{1/\la,L}\|_{\mpqd}}{\|D_{1/\la}F_{1/\la,L}\|_{\mpqd}}
\sim \la^{d(-2/p+1/q)},\ \ \la\in [1,\fy].
\end{split}
\ee
By the above estimates, we have
\be
\|D_{\la}\|_{\mtmd}\gtrsim \max\{\la^{-d/p}, \la^{d(1/q-1)}, \la^{d(-2/p+1/q)}\}=\G_{p,q}(\la)^d,
\ee
which is the desired conclusion.
\end{proof}

As mentioned in Section 1, Proposition \ref{pp-scalar} is an essential improvement of Theorem A.

\subsection{Diagonal matrix dilation}
In this subsection, the results of Proposition \ref{pp-scalar} will be extended to the case of diagonal matrix.
We remark that the upper bound estimates in this subsection coincide with that in Theorem B and the related preliminary results of \cite{CorderoNicola2007JoFA}. For the sake of coherence of the article, we give our own proof here.
In fact, most cases of diagonal matrix dilation can be reduced to the case of scalar matrix, except the two special cases considered in the following lemma.
See also case $(p,q)=(2,1)$ in the proof of \cite[Theorem 3.4]{CorderoNicola2007JoFA}.

\begin{lemma}\label{lm-diagonalc}
Let $\La=diag(\la_1,\cdots,\la_d)$. We have
  \ben\label{lm-diagonalc-cd1}
  \|D_{\La}\|_{M^{2,1}\rightarrow M^{2,1}}\lesssim 1,\ \ \ \text{if}\  \la_j\geq 1\ \text{for all}\ j=1,2,\cdots,d,
  \een
  and
  \ben\label{lm-diagonalc-cd2}
  \|D_{\La}\|_{M^{2,\fy}\rightarrow M^{2,\fy}}\lesssim |\det {\La}|^{-1},\ \ \ \text{if}\  \la_j\in (0,1]\ \text{for all}\ j=1,2,\cdots,d.
  \een
\end{lemma}
\begin{proof}
A direct calculation yields that
\be
\begin{split}
\Box_k(D_{\La}f)(x)
= &
\scrF^{-1}(\s_k(\xi)\widehat{D_{\La}f}(\xi))(x)
\\
= &
|\det {\La}|^{-1}\scrF^{-1}(\s_k(\xi)\widehat{f}(({\La}^T)^{-1}\xi))(x)
\\
= &
\scrF^{-1}((D_{{\La}^T}\s_k)\widehat{f})({\La}x).
\end{split}
\ee
Then, we have
  \ben\label{lm-diagonalc-1}
  \begin{split}
    \|D_{\La}f\|_{M^{2,1}}
    =
    \sum_{k\in \zd}\|\Box_k(D_{\La}f)\|_{L^2}
    =
    |\det {\La}|^{-1/2}\sum_{k\in \zd}\|\scrF^{-1}((D_{{\La}^T}\s_k)\widehat{f})\|_{L^2}.
  \end{split}
  \een
  For $l\in \zd$,
  recall that $Q_l$ denotes the unit cube with the center at $l$.
  Set
  \be
  \Theta_l=\{k\in \zd, k\in {\La}^T Q_l\}.
  \ee
  Observe that
  $\#\Theta_l\sim|\det {\La}|$.
  Using this, \eqref{lm-diagonalc-1}, and the Cauchy-Schwartz inequality, we conclude that
  \ben\label{lm-diagonalc-2}
  \begin{split}
    \|D_{\La}f\|_{M^{2,1}}
    \leq &
    |\det {\La}|^{-1/2}\sum_{l\in \zd}\sum_{k\in \Theta_l}\|\scrF^{-1}((D_{{\La}^T}\s_k)\widehat{f})\|_{L^2}
    \\
    \lesssim &
    |\det {\La}|^{-1/2}\sum_{l\in \zd}|\det {\La}|^{1/2}
    \bigg(\sum_{k\in \Theta_l}\|\scrF^{-1}((D_{{\La}^T}\s_k)\widehat{f})\|_{L^2}^2\bigg)^{1/2}
    \\
    = &
    \sum_{l\in \zd}\bigg(\int_{\rd}\sum_{k\in \Theta_l}|(D_{{\La}^T}\s_k(\xi))\widehat{f}(\xi)|^2 d\xi \bigg)^{1/2}.
  \end{split}
  \een
  Using the almost orthogonality of $\{\s_k\}_{k\in \zd}$, we have
  \be
  \sum_{k\in \Theta_l}|(D_{{\La}^T}\s_k)\widehat{f}|^2
  \sim
  \Big|\sum_{k\in \Theta_l}(D_{{\La}^T}\s_k)\widehat{f} \Big|^2.
  \ee
  Then, we continue the estimates of \eqref{lm-diagonalc-2} by
  \ben\label{lm-diagonalc-3}
  \|D_{\La}f\|_{M^{2,1}}
  \lesssim
  \sum_{l\in \zd} \bigg(\int_{\rd}\Big|\sum_{k\in \Theta_l}(D_{{\La}^T}\s_k(\xi))\widehat{f}(\xi)\Big|^2 d\xi \bigg)^{1/2}.
  \een
  By the definition of $\Theta_l$, we obtain that
  \be
  \text{supp}D_{{\La}^T}\s_k\subset ({\La}^T)^{-1}k+({\La}^T)^{-1}(2Q_0)\subset ({\La}^T)^{-1}k+2Q_0\subset Q_l+2Q_0=5Q_l.
  \ee
  From this, there exists a constant $N$ such that
  \be
  \sum_{k\in \Theta_l}(D_{{\La}^T}\s_k)\leq \sum_{|k-l|\leq N}\s_k.
  \ee
  Using this and \eqref{lm-diagonalc-3}, we have
  \be
  \begin{split}
  \|D_{\La}f\|_{M^{2,1}}
  \lesssim &
  \sum_{l\in \zd}\bigg(\int_{\rd} \Big|\sum_{|k-l|\leq N}\s_k(\xi)\widehat{f}(\xi) \Big|^2 d\xi \bigg)^{1/2}
  \\
  \lesssim &
  \sum_{l\in \zd}\sum_{|k-l|\leq N}\|\Box_kf\|_{L^2}
  \lesssim
   N^d\sum_{l\in \zd}\|\Box_lf\|_{L^2}
  =
  N^d\|f\|_{M^{2,1}}.
  \end{split}
  \ee
  We have now completed the proof of \eqref{lm-diagonalc-cd1}.
  Next, we turn to the proof of \eqref{lm-diagonalc-cd2}.
  Using \eqref{lm-diagonalc-1}, we have
  \ben\label{lm-diagonalc-4}
  \|D_{\La}f\|_{M^{2,\fy}}
  \leq
  |\det {\La}|^{-1/2}\sup_{k\in \zd}\|\scrF^{-1}((D_{{\La}^T}\s_k)\widehat{f})\|_{L^2}.
  \een
   Set
  \be
  \wt{\Theta}_k=\{l\in \zd, \s_lD_{{\La}^T}\s_k\neq 0\}.
  \ee
  Observe that $|\wt{\Theta}_k|\sim|\det {\La}|^{-1}$. Using this and the almost orthogonality of $\{\s_k\}_{k\in \zd}$, we conclude that
  \be
  \begin{split}
  \|\scrF^{-1}((D_{{\La}^T}\s_k)\widehat{f})\|_{L^2}
  = &
  \|(D_{{\La}^T}\s_k)\widehat{f}\|_{L^2}
  = 
  \Big\|\sum_{l\in \wt{\Theta}_k}\s_l(D_{{\La}^T}\s_k)\widehat{f}\Big\|_{L^2}
  \\
  \leq &
  \Big\|\sum_{l\in \wt{\Theta}_k}\s_l\widehat{f}\Big\|_{L^2}
  \lesssim |\det {\La}|^{-1/2}\sum_{l\in \wt{\Theta}_k}\|\s_l\widehat{f}\|_{L^2}
  \\
  \leq &
  |\det {\La}|^{-1/2}\|f\|_{M^{2,\fy}}.
  \end{split}
  \ee
  The final conclusion follows by this and \eqref{lm-diagonalc-4}.
\end{proof}

The following proposition give the asymptotic estimates of $\|D_{\La}\|_{\mtmd}$.

\begin{proposition}\label{pp-diagonal}
Let $1\leq p,q\leq \infty$, $\la_j>0$ for $j=1,2,\cdots,d$.
Denote ${\La}=diag(\la_1,\cdots,\la_d)$.
Then, there exist two constants $C_1$ and $C_2$ such that
\be
C_1\prod_{j=1}^d\G_{p,q}(\la_j)
\leq
\|D_{\La}\|_{\mtmd}\leq C_2\prod_{j=1}^d\G_{p,q}(\la_j).
\ee
\end{proposition}
\begin{proof}
  By Lemma \ref{lm-sp}, we have
  \be
  \begin{split}
  \|D_{\La}\|_{\mtmd}
  \geq  &
  \sup_{\|f_j\|_{\mpq}=1,1\leq j\leq d}\frac{\|D_{\La}(\otimes_{j=1}^df_j)\|_{\mpqd}}{\|\otimes_{j=1}^df_j\|_{\mpqd}}
  \\
  = &
  \sup_{\|f_j\|_{\mpq}=1,1\leq j\leq d}\frac{\|\otimes_{j=1}^dD_{\la_j}f_j\|_{\mpqd}}{\| \otimes_{j=1}^df_j\|_{\mpqd}}
  \\
  \sim &
  \prod_{j=1}^d
  \sup_{\|f_j\|_{\mpq}=1}\frac{\|D_{\la_j}f_j\|_{\mpq(\bbR)}}{\|f_j\|_{\mpq(\bbR)}}.
  \end{split}
  \ee
  Using this and Proposition \ref{pp-scalar}, we obtain that
  \be
  \begin{split}
  \|D_{\La}\|_{\mtmd}
  \gtrsim &
  \prod_{j=1}^d
  \sup_{\|f_j\|_{\mpq(\bbR)}=1}\frac{\|D_{\la_j}f_j\|_{\mpq(\bbR)}}{\|f_j\|_{\mpq(\bbR)}}
  \\
  = &
  \prod_{j=1}^d\|D_{\la_j}\|_{\mtm}
  \sim
  \prod_{j=1}^d\G_{p,q}(\la_j).
  \end{split}
  \ee
  Next, we turn to the estimate of upper bound.
  Let $\Phi=\otimes_{j=1}^d \phi$ with $\phi\in \calS(\rr)\setminus\{0\}$.
  Write
  \be
  V_{\Phi}(D_{\La}f)
  =
  \langle D_{\La}\Phi,D_{\La}\Phi \rangle^{-1}
   V_{\Phi}V^*_{D_{\La}\Phi}V_{D_{\La}\Phi}(D_{\La}f).
  \ee
  Noticing that
  \[
    \langle D_{\La}\Phi,D_{\La}\Phi \rangle
    =\int_{\bbR^d} |D_{\La}\Phi(x)|^2 dx
    =|\det {\La}|^{-1}\int_{\bbR^d} |\Phi(x)|^2 dx
    \sim
    |\det {\La}|^{-1},
  \]
 we have
  \be
  |V_{\Phi}(D_{{\La}}f)|
  \lesssim
  |\det {\La}||V_{\Phi}{D_{\La}\Phi}|\ast|V_{D_{\La}\Phi}(D_{\La}f)|.
  \ee
  Then, by Lemma \ref{lm-Young-mixed-norm}, we obtain that
  \ben\label{pp-diagonal-1}
  \begin{split}
  \|V_{\Phi}(D_{\La}f)\|_{\lpqsd}
  \lesssim &
  |\det {\La}|\|V_{\Phi}{D_{\La}\Phi}\|_{L^{1,1}(\bbR^{2d})}\|V_{D_{\La}\Phi}(D_{\La}f)\|_{\lpqsd}
  \\
  = &
  |\det {\La}|^{-1/p+1/q}\|V_{\Phi}{D_{\La}\Phi}\|_{L^{1,1}(\rdd)}\|f\|_{\mpqd},
  \end{split}
  \een
  where in the last equality we use the facts that
  \be
    V_{D_{\La}\Phi}(D_{\La}f)(x,\xi)=|\det {\La}|^{-1}V_{\Phi}f({\La}x,({\La}^T)^{-1}\xi),
  \ee
  and
  \begin{align*}
    \|V_{D_{\La}\Phi}(D_{\La}f)\|_{\lpqsd}
    &
    =|\det {\La}|^{-1}  \| V_{\Phi}f({\La}x,({\La}^T)^{-1}\xi)\|_{\lpqsd}
    \\
    &
    =|\det {\La}|^{-1-1/p+1/q}  \| V_{\Phi}f\|_{\lpqsd}
    \\
    &
    =|\det {\La}|^{-1-1/p+1/q}  \| f\|_{\mpqd}.
  \end{align*}
  Using Lemma \ref{lm-sp}, we have
  \be
  \|V_{\Phi}{D_{\La}\Phi}\|_{L^{1,1}(\rdd)}
  =
  \|D_{\La}\Phi\|_{M^{1,1}(\rd)}
  =
  \|\otimes_{j=1}^dD_{\la_j}\phi\|_{M^{1,1}(\rd)}
  =
  \prod_{j=1}^d\|D_{\la_j}\phi\|_{M^{1,1}(\rr)}.
  \ee

  From this and Theorem A, we continue the estimate of \eqref{pp-diagonal-1} by
  \ben\label{pp-diagonal-2}
  \begin{split}
  \|D_{\La}f\|_{\mpqd}
  = &
  \|V_{\Phi}(D_{{\La}}f)\|_{\lpqsd}
  \\
  \lesssim &
  |\det {\La}|^{-1/p+1/q}\prod_{j=1}^d\|D_{\la_j}\phi\|_{M^{1,1}(\rr)}\|f\|_{\mpqd}
  \\
  \lesssim &
  |\det {\La}|^{-1/p+1/q}\prod_{j=1}^d\G_{1,1}(\la_j)\|f\|_{\mpqd}.
  \end{split}
  \een
  Next, we consider three cases as follows.

  \textbf{Case 1.}
  If $\la_j\geq 1$ for all $j=1,2,\cdots,d$,
  we have $\G_{1,1}(\la_j)=1$. From this and the above estimate, we obtain that
  \ben\label{pp-diagonal-3}
  \|D_{\La}f\|_{\mpqd}\lesssim \prod_{j=1}^d\la_j^{-1/p+1/q}\|f\|_{\mpqd}.
  \een

  Observe that
  \be
  \prod_{j=1}^d\la_j^{-1/p+1/q}
  =\prod_{j=1}^d \G_{p,q}(\la_j)
  =\G_{p,q}\Big(\prod_{j=1}^d\la_j\Big)
  \ee
  for $p=\fy$, or $p=1$ or $q=\fy$.
  Using this, \eqref{pp-diagonal-3}, Lemma \ref{lm-diagonalc}, and the well-known conclusion for $\|D_{\La}\|_{M^{2,2}\rightarrow M^{2,2}}$, we conclude that
  \ben\label{pp-diagonal-4}
  \|D_{\La}\|_{\mtmd}
  \lesssim
  \G_{p,q}\Big(\prod_{j=1}^d\la_j\Big)
  \een
  is valid for $(p,q)=(2,2)$ or $(p,q)=(2,1)$, or $p=\fy$, or $p=1$ or $q=\fy$.
  By an interpolation argument, we deduce that \eqref{pp-diagonal-4} is valid for all $1\leq p,q\leq \fy$.

  \textbf{Case 2.}
  If $\la_j\in (0,1]$ for all $j=1,2,\cdots,d$, we have $\G_{1,1}(\la_j)=\la_j^{-1}$.
  Using this and \eqref{pp-diagonal-2}, we obtain that
  \be
  \|D_{\La}f\|_{\mpqd}\lesssim \prod_{j=1}^d\la_j^{-1/p+1/q-1}\|f\|_{\mpqd}.
  \ee
  We also have
  \be
  \prod_{j=1}^d\la_j^{-1/p+1/q-1}
  =\prod_{j=1}^d \G_{p,q}(\la_j)
  =\G_{p,q}\Big(\prod_{j=1}^d\la_j\Big)
  \ee
  for $p=\fy$, or $p=1$ or $q=1$.
  Using this, Lemma \ref{lm-diagonalc}, and the well-known conclusion for $\|D_{\La}\|_{M^{2,2}\rightarrow M^{2,2}}$, we conclude that
  \ben\label{pp-diagonal-5}
  \|D_{\La}\|_{\mtmd}
  \lesssim
  \G_{p,q}\Big(\prod_{j=1}^d\la_j\Big)
  \een
  is valid for $(p,q)=(2,2)$ or $(p,q)=(2,\fy)$, or $p=\fy$, or $p=1$ or $q=1$.
  By an interpolation argument, we deduce that \eqref{pp-diagonal-5} is valid for all $1\leq p,q\leq \fy$.

  \textbf{Case 3.}
  For general case, observe that
  \be
  \la_j=(\la_j\vee 1) (\la_j\wedge 1).
  \ee
  We write ${\La}={\La}_1{\La}_2$, with ${\La}_1=diag(\la_1\vee 1,\cdots,\la_d\vee 1)$
  and ${\La}_2=diag(\la_1\wedge 1,\cdots,\la_d\wedge 1)$.
  Using this and the conclusions in Case 1 and Case 2, we conclude that
  \be
  \begin{split}
  \|D_{\La}\|_{\mtmd}
  = &
  \|D_{{\La}_2}\circ D_{{\La}_1}\|_{\mtmd}
  \\
  \lesssim &
  \G_{p,q}\Big(\prod_{j=1}^d(\la_j\wedge 1)\Big) \cdot \G_{p,q}\Big(\prod_{j=1}^d(\la_j\vee 1)\Big)
  \\
  = &
  \prod_{j=1}^d\G_{p,q}(\la_j\wedge 1) \cdot \prod_{j=1}^d\G_{p,q}(\la_j\vee 1)
  =
  \prod_{j=1}^d\G_{p,q}(\la_j).
  \end{split}
  \ee
  We have now completed this proof.
\end{proof}

\subsection{General matrix case}\label{subsec. Dilation. General}
Now, we are in a position to deal with the most general case.
We recall that the singular values of $A$ mean all the eigenvalues of $\sqrt{A^TA}$, with each eigenvalue $\la$
repeated $dim E(\la, \sqrt{A^TA})$ times.

For a non-degenerate matrix $A\in GL(d,\rr)$, denote by $\{\la_j\}_{j=1}^d$  the singular values of $A$.
Let $\La=diag(\la_1,\cdots,\la_d)$.
Then there exist two orthogonal matrices, denoted by $P$ and $Q$ such that
\be
A=P{\La}Q.
\ee
Combing Lemma \ref{lm-orthogonal} and Proposition \ref{pp-diagonal}, we have
\be
\begin{split}
\|D_A\|_{\mtmd}
= &
\sup_{\|f\|_{\mpqd}=1}\frac{\|D_Af\|_{\mpqd}}{\|f\|_{\mpqd}}
\\
= &
\sup_{\|f\|_{\mpqd}=1}\frac{\|D_{P{\La}Q}f\|_{\mpqd}}{\|f\|_{\mpqd}}
\\
= &
\sup_{\|f\|_{\mpqd}=1}\frac{\|D_{P{\La}}f\|_{\mpqd}}{\|D_Pf\|_{\mpqd}}
\\
= &
\sup_{\|D_Pf\|_{\mpqd}=1}\frac{\|D_{{\La}}D_Pf\|_{\mpqd}}{\|D_Pf\|_{\mpqd}}
\\
=&
\|D_{\La}\|_{\mtmd}
\sim
\prod_{j=1}^d\G_{p,q}(\la_j).
\end{split}
\ee

\section{Hausdorff operators}
In this section, we explore the boundedness on modulation spaces of Hausdorff operators.
Before studying the  boundedness, let us first do some preliminary work.

\subsection{A revisit to the definition of Hausdorff operator}
As in \cite{ZhaoFanGuo2018AFA}, in order to study the boundedness of $H_{\Phi,A}$, we must first
clarify how the Hausdorff operator $H_{\Phi,A}$ acts on the functions in modulation spaces.
More precisely, we will explore the weakest assumption added on $\Phi$ to ensure that $H_{\Phi,A}f$ becomes a
tempered distribution for every $f\in \calS(\rd)$.

The general Hausdorff operator associated with non-degenerate matrices $A(y)$ can be formally defined as
\[
  H_{\Phi, A}f(x):=\int_{\rn}\Phi(y)f(A(y)x)dy,
\]
where the integral makes sense if $f$ belongs to a ``fine'' class of test functions such as the Schwartz spaces $\calS(\rd)$.
In this paper, if the boundedess on modulation spaces of $H_{\Phi,A}$ is valid, the most basic conditions should be that $H_{\Phi,A}$ is a continuous map from $\calS(\rd)$ into $\calS'(\rd)$.
In \cite{ZhaoFanGuo2018AFA}, we consider this problem for the case that $A(y)$ is a diagonal matrix with the entries $1/|y|$. Recently,
by a similar method in \cite{ZhaoFanGuo2018AFA},
 Karapetyants--Liflyand \cite{KarapetyantsLiflyand2020MMitAe} make a partial extension to the general Hausdorff operator $H_{\Phi,A}$.
 However, in fact the method in \cite{ZhaoFanGuo2018AFA,KarapetyantsLiflyand2020MMitAe} is not really suited for the general Hausdorff operator,
 which also leads some gaps in  \cite[Theorem 6]{KarapetyantsLiflyand2020MMitAe}.
 Here, we try to give a reasonable definition of Hausdorff operator such that the map
is continuous from $\calS(\rd)$ into $\calS'(\rd)$.

For $f,g \in \calS(\rd)$, if the map
\ben\label{df-H-1}
T_f: g\mapsto \int_{\rd}\int_{\rn}\Phi(y)f(A(y)x)\overline{g}(x)dydx
\een
becomes a continuous linear functional on $\calS(\rd)$, that is, belongs to $\calS'(\rd)$,
the Hausdorff operator can be defined weakly from $\calS(\rd)$ into $\calS'(\rd)$, by
\be
H_{\Phi,A}: f\mapsto T_f.
\ee
Using this definition, we have
\be
\langle H_{\Phi,A}f, g\rangle=\int_{\rd}\int_{\rn}\Phi(y)f(A(y)x)\overline{g}(x)dydx,
\ee
where the bracket in the left side means the extension to $\calS(\rd)\otimes\calS'(\rd)$ of the usual complex inner product in $L^2(\rd)$.
To describe clearly the conditions that make this definition work, we give the following theorem.
\begin{theorem}\label{thm-dfH}
  Suppose that $A(y)\in GL(d,\rr)$ is non-degenerate for all $y\in \rn$, with singular values
  $\{\la_j(y)\}_{j=1}^d$. If the following condition holds
  \begin{equation}\label{weakest cond.}
    	\int_{\rn} |\Phi(y)| \cdot  \prod_{j=1}^d (1\wedge\lambda_j(y)^{-1}) dy<\infty,
  \end{equation}
  then the map $\eqref{df-H-1}$ belongs to $\calS'(\rd)$, and the Hausdorff operator $H_{\Phi,A}$
  can be defined weakly as a map from $\calS(\rd)$ into $\calS'(\rd)$ as mentioned above.
  Conversely, if the function $\Phi$ is nonnegative, then \eqref{weakest cond.} is the weakest condition that makes this definition of $H_{\Phi,A}$ meaningful, that is, makes the map $\eqref{df-H-1}$ belong to $\calS'(\rd)$.
\end{theorem}
\begin{proof}
  We first verify that the map $\eqref{df-H-1}$ belongs to $\calS'(\rd)$ if \eqref{weakest cond.} holds.
  Let $f, g\in \calS(\rd)$. Observe that
  \be
  \begin{split}
  |f(x)|=|\langle x\rangle^{-\mathscr{L}}\langle x\rangle^{\mathscr{L}}f(x)|
  \leq \langle x\rangle^{-\mathscr{L}}\|\langle \cdot\rangle^{\mathscr{L}}f\|_{L^{\fy}(\rd)}.
  \end{split}
  \ee
  Using this and the the same estimate for $g$, we write
  \ben\label{df-H-2}
  \begin{split}
    &\bigg|\int_{\rd}\int_{\rn}\Phi(y)f(A(y)x)\overline{g}(x)dydx\bigg|
    \\
    \leq &
    \int_{\rd}\int_{\rn}\big|\Phi(y)\big|\langle A(y)x\rangle^{-\mathscr{L}}\langle x\rangle^{-\mathscr{L}}dydx
    \|\langle \cdot\rangle^{\mathscr{L}}f\|_{L^{\fy}(\rd)}\|\langle \cdot\rangle^{\mathscr{L}}g\|_{L^{\fy}(\rd)}
    \\
    = &
    \int_{\rn}\big|\Phi(y)\big|\bigg(\int_{\rd}\langle A(y)x\rangle^{-\mathscr{L}}\langle x\rangle^{-\mathscr{L}}dx\bigg)dy
    \|\langle \cdot\rangle^{\mathscr{L}}f\|_{L^{\fy}(\rd)}\|\langle \cdot\rangle^{\mathscr{L}}g\|_{L^{\fy}(\rd)}.
  \end{split}
  \een
  Recall that $A(y)=P(y)\La(y)Q(y)$ where $\La(y)=diag(\la_1(y),\cdots,\la_d(y))$,
   $P(y)$ and $Q(y)$ are two orthogonal matrices. Applying a change of variable, we deduce that
   \ben\label{df-H-3}
   \begin{split}
   \int_{\rd}\langle A(y)x\rangle^{-\mathscr{L}}\langle x\rangle^{-\mathscr{L}}dx
    = &
    \int_{\rd}\langle P(y)\La(y)Q(y)x\rangle^{-\mathscr{L}}\langle x\rangle^{-\mathscr{L}}dx
    = 
    \int_{\rd}\langle \La(y)x\rangle^{-\mathscr{L}}\langle x\rangle^{-\mathscr{L}}dx
    \\
    \sim &
    \int_{\rd}\prod_{j=1}^d\langle \la_j(y)x_j\rangle^{-\mathscr{L}}\prod_{j=1}^d\langle x_j\rangle^{-\mathscr{L}}dx
    =
    \prod_{j=1}^d\int_{\rr}\langle \la_j(y)x_j\rangle^{-\mathscr{L}}\langle x_j\rangle^{-\mathscr{L}}dx_j.
   \end{split}
   \een
Observe that
\be
\langle \la_j(y)x_j\rangle^{-\mathscr{L}}\langle x_j\rangle^{-\mathscr{L}}
\leq
\langle \la_j(y)x_j\rangle^{-\mathscr{L}}
\wedge
\langle x_j\rangle^{-\mathscr{L}},\ \ \ \ j=1,2,\cdots,d.
\ee
We obtain that
\ben\label{df-H-4}
\begin{split}
&
\int_{\rr}\langle \la_j(y)x_j\rangle^{-\mathscr{L}}\langle x_j\rangle^{-\mathscr{L}}dx_j
\\
\lesssim &
\int_{\rr}\langle x_j\rangle^{-\mathscr{L}}dx_j
\wedge
\int_{\rr}\langle \la_j(y)x_j\rangle^{-\mathscr{L}}dx_j
\lesssim
1\wedge \la_j(y)^{-1}.
\end{split}
\een
Using \eqref{df-H-2}, \eqref{df-H-3} and \eqref{df-H-4}, we find that
\be
\bigg|\int_{\rd}\int_{\rn}\Phi(y)f(A(y)x)\overline{g}(x)dydx\bigg|
\lesssim
\int_{\rn}\Phi(y) \cdot  \prod_{j=1}^d 1\wedge\lambda_j(y)^{-1} dy
\|\langle \cdot\rangle^{\mathscr{L}}f\|_{L^{\fy}(\rd)}\|\langle \cdot\rangle^{\mathscr{L}}g\|_{L^{\fy}(\rd)}.
\ee
Observe that
$\|\langle \cdot\rangle^{\mathscr{L}}f\|_{L^{\fy}(\rd)}$
and
$\|\langle \cdot\rangle^{\mathscr{L}}g\|_{L^{\fy}(\rd)}$
can be dominated by certain $\calS(\rd)$ semi-norms of $f$ and $g$, respectively.
Thus, we conclude that for every $f\in\calS(\rd)$, the map $T_f$ in $\eqref{df-H-1}$ belongs to $\calS'(\rd)$, and the Hausdorff operator $H_{\Phi,A}$
  can be defined weakly by $H_{\Phi,A}: f\mapsto T_f$
  as a continuous map from $\calS(\rd)$ into $\calS'(\rd)$.

  Next, we will prove the optimality of condition \eqref{weakest cond.}.
  Let $f,g\in \calS(\rd)$ with $f, g\geq \chi_{B(0,\sqrt{d})}$. Recall that $\Phi$ is non-negative, then
  \be
  \begin{split}
    \fy
    >&
    \int_{\rd}\int_{\rn}\Phi(y)f(A(y)x)\overline{g}(x)dydx
    \geq 
    \int_{\rn}\Phi(y)\int_{B(0,\sqrt{d})}\chi_{B(0,\sqrt{d})}(A(y)x)dx dy
    \\
    = &
    \int_{\rn}\Phi(y)|\det A(y)|^{-1}\int_{A(y)B(0,\sqrt{d})}\chi_{B(0,\sqrt{d})}(x)dx dy
    \\
    = &
    \int_{\rn}\Phi(y)|\det A(y)|^{-1}|A(y)B(0,\sqrt{d})\cap B(0,\sqrt{d})|dy.
  \end{split}
  \ee
  Recall $A(y)=P(y)\La(y)Q(y)$ as mentioned above. We conclude that

	\begin{align*}
	  |A(y)(B(0,\sqrt{d}))\cap B(0,\sqrt{d})|
	  &=| P(y)\La(y)Q(y)(B(0,\sqrt{d}))\cap B(0,\sqrt{d}) |
	  \\
	  &=|     \La(y)Q(y)(B(0,\sqrt{d}))\cap B(0,\sqrt{d}) |
	  \\
	  &=|     \La(y)      (B(0,\sqrt{d}))\cap B(0,\sqrt{d}) |
	  \\
	  &\geq
	    |\La(y)([-1/2,1/2]^d)\cap [-1/2,1/2]^d|
      \\
      &=
	    \prod_{j=1}^d 1\wedge\lambda_j(y).
	\end{align*}
	Using the above two estimates, we obtain
	\begin{align*}
	\infty
	&>
    \int_{\rn}\Phi(y)|\det A(y)|^{-1}|A(y)B(0,\sqrt{d})\cap B(0,\sqrt{d})|dy
    \\
    &\geq
	\int_{\rn}\Phi(y) \cdot \prod_{j=1}^d \lambda_j(y)^{-1} \cdot \prod_{j=1}^d 1\wedge\lambda_j(y) dy
	=
	\int_{\rn}\Phi(y) \cdot  \prod_{j=1}^d 1\wedge\lambda_j(y)^{-1} dy.
	\end{align*}
	This proves the optimality of the condition (\ref{weakest cond.}).
    \end{proof}

    Next, we want to show that the adjoint operator of $H_{\Phi,A}$ can also be well defined if  (\ref{weakest cond.}) holds.

    \begin{proposition}[Adjoint operator of Hausdorff operators]\label{R. Adj. Hausdorff}
    Suppose that $A(y)\in GL(d,\rr)$ is non-degenerate for all $y\in \rn$, with singular values
  $\{\la_j(y)\}_{j=1}^d$, and
    $\Phi$ satisfies the condition (\ref{weakest cond.}).
    Let $\wt{\Phi}(y)=\overline{\Phi(y)}\prod_{j=1}^d\la_j(y)^{-1}$.
    The adjoint Hausdorff operator
    $H^*_{\Phi,A}:=H_{ \wt{\Phi},A^{-1}}$ can be defined as a continuous map from
    $\calS(\rd)$ into $\calS'(\rd)$, by
    \be
    \langle H_{ \wt{\Phi},A^{-1}}f,g\rangle
    =
    \int_{\rd}\int_{\rn}\wt{\Phi}(y)f(A^{-1}(y)x)\overline{g}(x)dydx,\ \ \ \ \ f,g \in \calS(\rd).
    \ee
    Moreover, we have the adjoint relation between $H_{\Phi,A}$ and $H_{\wt{\Phi}, A^{-1}}$:
    \be
    \langle H_{\Phi, A}f, g \rangle=\overline{\langle H_{\wt{\Phi}, A^{-1}}g, f \rangle},\ \ \ \ \ f,g \in \calS(\rd).
    \ee
    \end{proposition}

 \begin{proof}
    Observe that
    $\{\la_j(y)^{-1}\}_{j=1}^d$ are precisely the singular values of $A^{-1}$, and
    \be
    \prod_{j=1}^d \big(1\wedge\lambda_j(y)^{-1}\big)
    =\prod_{j=1}^d\la_j(y)^{-1}\cdot \prod_{j=1}^d \big(1\wedge\lambda_j(y)\big).
    \ee
    Then,
    \be
    \int_{\rn}|\Phi(y)| \cdot  \prod_{j=1}^d \big(1\wedge\lambda_j(y)^{-1}\big) dy<\infty
    \Longleftrightarrow
    \int_{\rn}|\wt{\Phi}(y)| \cdot  \prod_{j=1}^d \big(1\wedge\lambda_j(y)\big) dy<\infty.
    \ee
    From this, the condition (\ref{weakest cond.}) also holds if
    we replace $\Phi$ and $\la_j$ by
    $\wt{\Phi}$ and $\la_j^{-1}$ respectively.
    Thus, $H_{ \wt{\Phi},A^{-1}}$ can be well defined.

Moreover, a direct calculation shows that
  \begin{align*}
    \langle H_{\Phi, A}f, g \rangle
    &=
     \int_{\rn}\Phi(y)\int_{\bbR^d} f(A(y)x) \cdot \overline{g}(x)dx dy
    \\
    &=
    \int_{\rn} \Phi(y)  \cdot |\det{A^{-1}(y)}| \cdot \int_{\bbR^d}  f(x) \cdot \bar{g}(A^{-1}(y)x) dxdy
    \\
    &=
    \int_{\rd}\int_{\rn}\Phi(y)  \cdot \prod_{j=1}^d\la_j(y)^{-1} \cdot  \bar{g}(A^{-1}(y)x)f(x)dydx
    \\
    &=
    \overline{\int_{\rd}\int_{\rn}\overline{\Phi(y)}  \cdot \prod_{j=1}^d\la_j(y)^{-1} \cdot  g(A^{-1}(y)x)\overline{f(x)}dydx}
    =
    \overline{\langle H_{ \wt{\Phi},A^{-1}}g, f \rangle}.
  \end{align*}
    \end{proof}

As follows, we will show that under the condition \eqref{weakest cond.},
the (partial) Fourier transform of $H_{\Phi,A}$ can be well defined in the sense of distributions.

Following the symbol in \cite{GrochenigBook2013},
we adopt the following partial Fourier and inverse Fourier transform for $f\in\scrS$:
\begin{align*}
\fj f(x)     =\int_{\bbR} f(x_1,\cdots,y_j,\cdots,x_d) \cdot e^{-2\pi i y_j\cdot x_j} dy_j,
\\
\fj^{-1} f(x)=\int_{\bbR} f(x_1,\cdots,y_j,\cdots,x_d) \cdot e^{2\pi i y_j\cdot x_j} dy_j.
\end{align*}
For $\mathbb{J}=\{j_1,j_2,\cdots,j_m\}\subset \{1,2,\cdots, d\}$ with $m\geq1$, we define
\begin{align*}
\fJ f(x):=\scrF_{j_1}\cdots\scrF_{j_m}f(x),
\hspace{1mm}
\fJ^{-1} f(x):=\scrF_{j_1}^{-1}\cdots\scrF_{j_m}^{-1}f(x).
\end{align*}
Especially,  $\fJ f=\scrF f$ when $\mathbb{J}=\{1,2,\cdots,d\}$.
And we denote $\fJ f=\fJ^{-1} f=f$ when $\mathbb{J}=\emptyset$.

\begin{proposition}\label{pp-Fourier-Hausdorff}
  Suppose that $A(y)\in GL(d,\rr)$ is non-degenerate for all $y\in \rn$, with singular values
  $\{\la_j(y)\}_{j=1}^d$, and
    $\Phi$ satisfies the condition (\ref{weakest cond.}).
    Let $\Phi_{\bbJ}(y)={\Phi(y)}\prod_{j\in \bbJ}\la_j(y)^{-1}$.
    Then the Fourier transform can be well defined in the sense of distributions by
    \be
    \lan \fJ H_{\Phi,A}f, g\ran=\lan H_{\Phi,A}f, \fJ^{-1}g\ran,\ \ \ \ \ f,g \in \calS(\rd).
    \ee
    Moreover, we have the following equation in the sense of distributions
    \be
    \scrF H_{\Phi,A}f=H_{\Phi_{\bbJ},(A^{T})^{-1}}\scrF f,\ \ \ \ \bbJ=\{1,2,\cdots,d\},\ \ \ f \in \calS(\rd).
    \ee
    If $A(y)=\La(y)=\text{diag}\{\lambda_1(y),\cdots,\lambda_d(y) \}$,
    for all $\bbJ\subset \{1,2,\cdots,d\}$
    we have the following equation in the sense of distributions
    \be
    \fJ H_{\Phi,\La}f=H_{\Phi_{\bbJ},\La_{\bbJ}}\fJ f, \ \ \ f \in \calS(\rd).
    \ee
    where
  $\La_{\mathbb{J}}=\text{diag}\{\g_1(y), \cdots, \g_d(y)\}$ with
  \begin{equation*}
 \g_k(y)=
 \begin{cases}
 \la^{-1}_j(y),  &\text{if } j\in  \mathbb{J},\\
 \la_j(y), &\text{if }  j\notin\mathbb{J}.
 \end{cases}
 \end{equation*}
\end{proposition}
\begin{proof}
  As in the proof of Proposition \ref{R. Adj. Hausdorff},  it is not difficult to verify
  that the corresponding condition as in \eqref{weakest cond.} is valid for $H_{\Phi_{\bbJ},(A^{T})^{-1}}$ and $H_{\Phi_{\bbJ},\La_{\bbJ}}$.
  Thus, as two special Hausdorff operators,  $H_{\Phi_{\bbJ},(A^{T})^{-1}}$ and $H_{\Phi_{\bbJ},\La_{\bbJ}}$ are well defined.
  
  For any $f,g\in \calS(\rd)$, $\bbJ=\{1,2,\cdots,d\}$,
 a direct calculation yields that
  \be
  \begin{split}
    \lan \scrF H_{\Phi,A}f, g\ran
    = &
    \lan H_{\Phi,A}f, \scrF^{-1}g\ran
    \\
    = &
    \int_{\rn}\Phi(y)\int_{\rd}f(A(y)x)\overline{\scrF^{-1}g(x)}dxdy
    \\
    = &
    \int_{\rn}\Phi(y)\int_{\rd}\scrF (f(A(y)\cdot))(x)\overline{g(x)}dxdy
        \\
    = &
    \int_{\rn}\Phi(y)\int_{\rd}|\det A(y)|^{-1}
    (\scrF f)((A(y)^T)^{-1}x)\overline{g(x)}dxdy
    \\
    = &
    \lan H_{\Phi_{\bbJ},(A^{T})^{-1}}\scrF f, g\ran.
  \end{split}
  \ee
  If $A(y)=\La(y)$, using the same method as above, we obtain that for all $\bbJ\subset \{1,2,\cdots,d\}$,
  \be
  \lan \fJ H_{\Phi,\La}f, g\ran
  =
  \int_{\rn}\Phi(y)\int_{\rd} \fJ (f(\La(y)\cdot))(x)\overline{g(x)}dxdy.
  \ee
  The desired conclusion follows by
  \be
  \begin{split}
  \fJ(f(\La(y)\cdot))(x)
  = &
  (\fJ f)(\La_{\bbJ}(y)x)
  \prod_{j\in\mathbb{J}} \la_j(y)^{-1}.
     \end{split}
  \ee
\end{proof}

Now, we are in a position to consider the boundedness property of $H_{\Phi,A}$ on modulation spaces.

\subsection{The boundedness of Hausdorff operators}
We turn to give the boundedness of Hausdorff operators.
First we establish the following general version of fundamental identity of time frequency analysis.
See the usual version in \eqref{itf}.

For $\mathbb{J}\subset \{1,2,\cdots, d\}$ and $x=(x_1,x_2,\cdots,x_d)$, denote
$\mathbb{J}^C$ the complementary set of $\mathbb{J}$ in $\{1,2,\cdots, d\}$,
and $x_{\mathbb{J}}:=\{y_1,y_2,\cdots, y_d\}$, where
 \begin{equation*}
 y_j=
 \begin{cases}
 x_j,  &\text{if } j\in  \mathbb{J},\\
 0, &\text{if }  j\notin\mathbb{J}.
 \end{cases}
 \end{equation*}
\begin{proposition}
\label{pp-fund.id.TF.partF}
Let $\mathbb{J}\subset \{1,2,\cdots, d\}$, the window function $g\in\scrS\setminus\{0\}$, then
	\[
	V_gf(x,w)
	=
	e^{-2\pi i x_{\mathbb{J}}\cdot w_{\mathbb{J}}}
	\cdot
	V_{\fJ g} \fJ f(w_\mathbb{J}+x_{\mathbb{J}^C}, -x_\mathbb{J}+w_{\mathbb{J}^C}).
	\]
\end{proposition}
\begin{proof}
  When $\mathbb{J}=\emptyset$, it holds obviously, and when $\mathbb{J}=\{1,2,\cdots, d\}$, it is the usual fundamental identity of time frequency.

 Observing that $x=x_{\mathbb{J}}+x_{\mathbb{J}^C}$, $T_x=T_{x_\mathbb{J}} T_{x_{\mathbb{J}^C}}$ and $M_w=M_{w_\mathbb{J}} M_{w_{\mathbb{J}^C}}$,
  for $x=(x_1,x_2,\cdots,x_d)$ and $w=(w_1,w_2,\cdots,w_d)$, we have
  \begin{align*}
     \fJ( M_{w}T_x g)
    &=
     \fJ( M_{w_\mathbb{J}} M_{w_{\mathbb{J}^C}} T_{x_\mathbb{J}} T_{x_{\mathbb{J}^C}} g)
    =
     \fJ( M_{w_\mathbb{J}} T_{x_\mathbb{J}} M_{w_{\mathbb{J}^C}}  T_{x_{\mathbb{J}^C}} g)
    \\
    &=
     e^{2\pi i x_{\mathbb{J}}\cdot w_{\mathbb{J}}}
     \cdot
     M_{-x_\mathbb{J}} T_{w_\mathbb{J}} \fJ( M_{w_{\mathbb{J}^C}}  T_{x_{\mathbb{J}^C}} g)
   \\
   &=
     e^{2\pi i x_{\mathbb{J}}\cdot w_{\mathbb{J}}}
     \cdot
     M_{-x_\mathbb{J}} T_{w_\mathbb{J}} M_{w_{\mathbb{J}^C}}  T_{x_{\mathbb{J}^C}} \fJ g
   \\
   &=
     e^{2\pi i x_{\mathbb{J}}\cdot w_{\mathbb{J}}}
     \cdot
     M_{-x_\mathbb{J}} M_{w_{\mathbb{J}^C}} T_{w_\mathbb{J}}  T_{x_{\mathbb{J}^C}} \fJ g
   \\
   &=
     e^{2\pi i x_{\mathbb{J}}\cdot w_{\mathbb{J}}}
     \cdot
     M_{-x_\mathbb{J}+w_{\mathbb{J}^C}} T_{w_\mathbb{J}+x_{\mathbb{J}^C}} \fJ g.
  \end{align*}
  Draw support from this, we have
  \begin{align*}
    V_gf(x,w)
    &=
    \langle f, M_{w}T_x g \rangle
    =
    \langle \fJ f, \fJ( M_{w}T_x g) \rangle
    \\
    &=
     e^{-2\pi i x_{\mathbb{J}}\cdot w_{\mathbb{J}}}
     \cdot
     V_{\fJ g} \fJ f(w_\mathbb{J}+x_{\mathbb{J}^C}, -x_\mathbb{J}+w_{\mathbb{J}^C}).
  \end{align*}
\end{proof}
In order to establish the estimates of Hausdorff operator from below, we introduce the following partial Fourier modulation space.
Subsequently we will give the embedding relation between this space and the mixed Lebesgue space.

\begin{definition}[The partial Fourier modulation space]
	Let $0<p,q\leq \infty$.
	The function space $\fJmpqd$, which we call the partial Fourier modulation space,   consists of all tempered distributions $f\in\scrS'$ such that
	\[
	\|f\|_{\fJmpqd}:=\|\fJ^{-1} f\|_{\mpqd}
	\]
	is finite.
\end{definition}

\begin{proposition}[Embedding between $\fJmpq$ and $\lqp$]\label{pp-embed-fjmpq-lpq}
  Let $1/2\leq 1/p\leq 1/q\leq 1$ and $\mathbb{J}\subset\{1,2,\cdots,d\}$, we have $\fJmpqd \hookrightarrow\lqp(\bbR^J\times\bbR^{d-J})$.
\end{proposition}
\begin{proof}
We only give the proof for  nonempty $\mathbb{J}$ which is a proper subset of  $\{1,2,\cdots,d\}$, for other cases are similar.

To begin with,  we deal with the special case that $\mathbb{J}=\{1,2,\cdots,J\}$.
Let nonzero functions $\varphi_1\in C_c^{\infty}(\bbR^{J})$, $\varphi_2\in C_c^{\infty}(\bbR^{d-J})$.
Using Proposition \ref{pp-fund.id.TF.partF},  we have
  \begin{align*}
    \|f\|_{\fJmpq}
    &=
      \|\scrF_{\mathbb{J}}^{-1} f\|_{\mpq}
      =
      \|
        V_{\breve{\varphi}_1\otimes\varphi_2} \big( \scrF_{\mathbb{J}}^{-1} f\big)(\bx_1,\bx_2,\bxi_1,\bxi_2)
      \|_{L_{\bx_1,\bx_2,\bxi_1,\bxi_2}^{p,p,q,q}}
    \\
    &=
     \|
       V_{\varphi_1\otimes\varphi_2} f(\bxi_1,\bx_2,-\bx_1,\bxi_2)
     \|_{L_{\bx_1,\bx_2,\bxi_1,\bxi_2}^{p,p,q,q}}
     \\
     &=
     \Bigg\|
       \scrF_{\mathbb{J}}
        \bigg(
          \scrF_{\mathbb{J^C}}
           \big(
            f(\by_1,\by_2)\overline{\varphi_2(\by_2-\bx_2)}
           \big) (\bxi_2)
           \cdot
           \overline{\varphi_1(\by_1-\bxi_1)}
        \bigg)(-\bx_1)
     \Bigg\|_{L_{\bx_1,\bx_2,\bxi_1,\bxi_2}^{p,p,q,q}}
    \\
    &=
    \Bigg\|
      \scrF_{\mathbb{J}}
      \bigg(
       \scrF_{\mathbb{J^C}}
       \big(
       f(\by_1,\by_2)\overline{\varphi_2(\by_2-\bx_2)}
      \big) (\bxi_2)
      \cdot
      \overline{\varphi_1(\by_1-\bxi_1)}
     \bigg)(\bx_1)
   \Bigg\|_{L_{\bx_1,\bx_2,\bxi_1,\bxi_2}^{p,p,q,q}}.
  \end{align*}
  Noticing that $\varphi_1$ and $\varphi_2$ have compact supports in $\bbR^J$ and $\bbR^{d-J}$ respectively,
  and $1/2\leq 1/p\leq 1/q\leq 1$ which yields that $1/q'\leq1/p'\leq 1/2\leq 1/p\leq 1/q\leq 1$,
  we have
  \begin{align*}
    \|f\|_{\fJmpq}
    &\geq
     \Bigg\|
          \scrF_{\mathbb{J^C}}
           \big(
            f(\bx_1,\by_2) \cdot \overline{\varphi_2(\by_2-\bx_2)}
           \big) (\bxi_2)
           \cdot
           \overline{\varphi_1(\bx_1-\bxi_1)}
     \Bigg\|_{L_{\bx_1,\bx_2,\bxi_1,\bxi_2}^{p',p,q,q}}
    \\
    &\geq
         \Bigg\|
              \scrF_{\mathbb{J^C}}
               \big(
                f(\bx_1,\by_2) \cdot \overline{\varphi_2(\by_2-\bx_2)}
               \big) (\bxi_2)
               \cdot
               \overline{\varphi_1(\bx_1-\bxi_1)}
         \Bigg\|_{L_{\bxi_2,\bx_1,\bx_2,\bxi_1}^{q,p',p,q}}
    \\
    &\geq
        \Bigg\|
               f(\bx_1,\bxi_2) \cdot \overline{\varphi_2(\bxi_2-\bx_2)}
               \cdot
               \overline{\varphi_1(\bx_1-\bxi_1)}
       \Bigg\|_{L_{\bxi_2,\bx_1,\bx_2,\bxi_1}^{q',p',p,q}}
    \\
    &\geq
        \Bigg\|
               f(\bx_1,\bxi_2) \cdot \overline{\varphi_2(\bxi_2-\bx_2)}
               \cdot
               \overline{\varphi_1(\bx_1-\bxi_1)}
       \Bigg\|_{L_{\bxi_2,\bx_2,\bx_1,\bxi_1}^{q',p,p',q}}
    \\
    &\gtrsim
        \Bigg\|
               f(\bx_1,\bxi_2) \cdot \overline{\varphi_2(\bxi_2-\bx_2)}
               \cdot
               \overline{\varphi_1(\bx_1-\bxi_1)}
       \Bigg\|_{L_{\bxi_2,\bx_2,\bx_1,\bxi_1}^{p,p,q,q}}
    \\
    &\sim
    \|f(\bx_1, \bxi_2)\|_{l^{p,q}_{\bxi_2,\bx_1}}
     \geq
      \|f(\bx_1, \bxi_2)\|_{\lqp_{\bx_1, \bxi_2}}
    =
    \|f\|_{\lqp(\bbR^J\times \bbR^{d-J})},
  \end{align*}
  where we use the Hausdorff-Young and the H\"{o}lder inequalities in the above estimates.

  Now we turn to the general case. For any $\mathbb{J}=\{j_1, j_2,\cdots,j_J\}\subset \{1,2,\cdots,d\}$,
  there is an orthogonal matrix $P$ such that
  \[
    P(x_1,x_2,\cdots,x_d)
    =(x_{j_1}, x_{j_2}, \cdots, x_{j_J},x_{i_1},\cdots, x_{i_{d-J}}),
   \]
  where $\{i_1,\cdots, i_{d-J} \}=\mathbb{J}^C$.
  Set
  $g(x)=(f\circ P^{-1}) (x).$
  We have
  \be
  g(x_{j_1}, x_{j_2}, \cdots, x_{j_J},x_{i_1},\cdots, x_{i_{d-J}})=f(x_1,x_2,\cdots,x_d).
  \ee
  Denote $\wt{\mathbb{J}}=\{1,2,\cdots J\}$.
  A direct calculation yields that
  \be
  \begin{split}
  \fJ^{-1} f(\xi)
  = &
  (\scrF^{-1}_{\wt{\mathbb{J}}}g)(P\xi).
  \end{split}
  \ee
  Using this, Lemma \ref{lm-orthogonal}, and the corresponding conclusion of special case proved above, we deduce that
  \begin{align*}
    \| f \|_{\fJmpqd}
    &= 
    \| \fJ^{-1} f \|_{\mpqd}
    = 
    \| (\scrF^{-1}_{\wt{\mathbb{J}}}g) \circ P\|_{\mpqd}
     =
      \| \scrF^{-1}_{\wt{\mathbb{J}}}g\|_{\mpqd}
     \\
     &\gtrsim
      \| g \|_{\lqp(\bbR^J\times\bbR^{d-J})}\gtrsim \| f \|_{\lqp(\bbR^J\times\bbR^{d-J})}.
  \end{align*}

\end{proof}

Next, we give the key estimate for proving the necessity of the boundedness of Hausdorff operator on modulation space.

\begin{proposition}\label{pp-FJM-Lqp}
   Let $1/2\leq 1/p\leq 1/q\leq 1$ and $\mathbb{J}\subset\{1,2,\cdots,d\}$.
   Suppose that the basic condition \eqref{weakest cond.} holds.
   Let $\Phi_{\bbJ}(y)={\Phi(y)}\prod_{j\in \bbJ}\la_j(y)^{-1}$,
   and let $\La_{\mathbb{J}}=\text{diag}\{\g_1(y), \cdots, \g_d(y)\}$ with
   \begin{equation*}
   	\g_k(y)=
   	\begin{cases}
   		\la^{-1}_j(y),  &\text{if } j\in  \mathbb{J},\\
   		\la_j(y), &\text{if }  j\notin\mathbb{J}.
   	\end{cases}
   \end{equation*}
   Then, the following boundedness
    $$H_{\Phi_{\bbJ},\La_{\bbJ}}: \fJmpq(\bbR^d) \rightarrow \lqp(\bbR^J\times\bbR^{d-J})$$
   implies that
   \begin{equation}\label{pp-FJM-Lqp-condition}
     \int_{\rn}\Phi(y) \prod_{j\in\mathbb{J}} \lambda_j(y)^{1/q-1}  \prod_{j\notin\mathbb{J}} \lambda_j(y)^{-1/p} dy<\infty.
   \end{equation}
\end{proposition}
\begin{proof}
  Let $M, N$ be two sufficiently large positive numbers with $M>2N$.
  Choose a nonnegative function $\eta\in\calS(\rd)$ with $\eta(0)=1$,
  satisfying that $\scrF\eta$ is supported on $B(0,1)$.
  Denote  $F_M:=[2^{-1},2^{M+1}]$ and $\wt{F}_M:=[1,2^{M}]$.
  Let
  $
    f(x):=\prod_{j=1}^d f_j(x_j)
  $, where
  \begin{equation*}
  f_j(x)=
  \begin{cases}
  \big( |\cdot|^{-1/q}\chi_{F_M}(\cdot) \big) \ast\eta,  &j\in\mathbb{J},\\
  \big( |\cdot|^{-1/p}\chi_{F_M}(\cdot) \big) \ast\eta,  &j\notin\mathbb{J}.
  \end{cases}
  \end{equation*}

  Noticing that $f_j\in \calS(\rr)$ and
   $\scrF{f_j}$ also has compact support, $j=1, 2, \cdots, d$,
  and
  by Lemmas \ref{lm-sp}, \ref{lm-lp-1} and \ref{lm-lp-2}, we have
  \begin{equation}\label{pp-FJM-Lqp-right}
  \begin{split}
    \|f\|_{\fJmpq(\bbR^d)}
    &=
    \Big\| \prod_{j=1}^d f_j(x_j) \Big\|_{\fJmpq(\bbR^d)}
     =
     \Big\|\prod_{j\in\mathbb{J}} \scrF^{-1}f_j(x_j)
      \cdot
      \prod_{j\notin\mathbb{J}} f_j(x_j) \Big\|_{\mpq(\bbR^d)}
    \\
    &=
    \prod_{j\in\mathbb{J}} \|\scrF^{-1}f_j\|_{\mpq(\bbR)}
     \cdot
     \prod_{j\notin\mathbb{J}} \| f_j\|_{\mpq(\bbR)}
    \\
    &\sim
    \prod_{j\in\mathbb{J}} \|f_j\|_{L^q(\bbR)}
     \cdot
     \prod_{j\notin\mathbb{J}} \| f_j\|_{L^p(\bbR)}
    \\
    &\lesssim
    \prod_{j\in\mathbb{J}} \||\cdot|^{-1/q}\chi_{F_M}(\cdot)\|_{L^q(\bbR)}
     \cdot
     \prod_{j\notin\mathbb{J}} \| |\cdot|^{-1/p}\chi_{F_M}(\cdot)\|_{L^p(\bbR)}
    \\
    &\sim
    (M+2)^{J/q+(d-J)/p}.
  \end{split}
  \end{equation}

   Next, we turn to the lower estimate of $H_{\Phi_{\bbJ},\La_{\bbJ}}$.
   Observing that $\eta(0)=1$, there is a small positive constant $\d$ such that $\eta(x)\geq1/2$ for all $|x|<\delta$.
   For $r\in\bbR$, a direct calculation yields that
  \begin{align*}
    \Big(\big( |\cdot|^r\chi_{F_M}(\cdot) \big) \ast\eta\Big) (x)
    &=
    \int_{\bbR} |t|^r \cdot \chi_{F_M}(t) \cdot \eta(x-t)dt
    \\
    &\gtrsim
    \int_{B(x,\delta)} |t|^rdt
    \gtrsim
    |x|^r \cdot \chi_{\wt{F}_M}(x).
  \end{align*}
  Using this, 
  we conclude that
  \begin{align*}
    \|H_{\Phi_{\bbJ},\La_{\bbJ}} f\|_{ \lqp(\bbR^J\times\bbR^{d-J}) }
     =&
     \Big\|
     \int_{\rn}
      \Phi(y) \cdot \prod_{j\in\mathbb{J}} \lambda_j(y)^{-1}
      \cdot
      \prod_{j\in\mathbb{J}} \Big( \big( |\cdot|^{-1/q}\chi_{F_M}(\cdot) \big) \ast\eta \Big)(\lambda_j^{-1}(y)x_j)
      \\
      &\cdot
      \prod_{j\notin\mathbb{J}} \Big( \big( |\cdot|^{-1/p}\chi_{F_M}(\cdot) \big) \ast\eta \Big)(\lambda_j(y)x_j)
     dy
     \Big\|_{ \lqp(\bbR^J\times\bbR^{d-J}) }
    \\
    \gtrsim&
    \Big\|
     \int_{E_N}
      \Phi(y)
      \cdot \prod_{j\in\mathbb{J}} \lambda_j(y)^{1/q-1}
      \cdot
      \prod_{j\notin\mathbb{J}}  \lambda_j(y)^{-1/p}
      \cdot
      \prod_{j\in\mathbb{J}}   |x_j|^{-1/q}   \chi_{\wt{F}_M}(\lambda_j^{-1}(y)x_j)
      \\
      &\cdot
      \prod_{j\notin\mathbb{J}}  |x_j|^{-1/p}  \chi_{\wt{F}_M}(\lambda_j(y)x_j)
     dy
     \Big\|_{ \lqp(\bbR^J\times\bbR^{d-J}) },
  \end{align*}
  where $E_N:=\{y\in\rn: \lambda_j(y)\in[2^{-N}, 2^N], j=1, 2, \cdots, d\}$.
  Denote
  \begin{equation*}
  G_{M,N}=[2^N,    2^{M-N}].
  \end{equation*}
  For any $y\in E_N$ and $x_j\in G_{M,N}$, we have
  $\lambda_j^{-1}(y)x_j\in \wt{F}_M$ when $j\in\mathbb{J}$, and
  $\lambda_j(y)x_j\in \wt{F}_M$ when $j\notin\mathbb{J}$.
  Then we have
  \begin{equation}\label{pp-FJM-Lqp-left}
  \begin{split}
    \|H_{\Phi_{\bbJ},\La_{\bbJ}} f\|_{ \lqp(\bbR^J\times\bbR^{d-J}) }
     \gtrsim&
    \Big\|
     \int_{E_N}
      \Phi(y)
      \cdot \prod_{j\in\mathbb{J}} \lambda_j(y)^{1/q-1}
      \cdot
      \prod_{j\notin\mathbb{J}}  \lambda_j(y)^{-1/p}
      \cdot
      \prod_{j\in\mathbb{J}}   |x_j|^{-1/q}
      \\
      &\cdot
      \prod_{j\notin\mathbb{J}}  |x_j|^{-1/p}
     dy
     \Big\|_{ \lqp(G_{M,N}^{J}\times G_{M,N}^{d-J}) }
     \\
     =&
     \int_{E_N}
      \Phi(y)
      \cdot \prod_{j\in\mathbb{J}} \lambda_j(y)^{1/q-1}
      \cdot
      \prod_{j\notin\mathbb{J}}  \lambda_j(y)^{-1/p}
      dy
      \\
      &\cdot
      \prod_{j\in\mathbb{J}}
      \|
      |x_j|^{-1/q}
      \|_{ L^q(G_{M,N}) }
      \cdot
      \prod_{j\notin\mathbb{J}}
      \|
       |x_j|^{-1/p}
      \|_{ L^p(G_{M,N}) }
     \\
     \sim&
     \int_{E_N}
      \Phi(y)
      \cdot \prod_{j\in\mathbb{J}} \lambda_j(y)^{1/q-1}
      \cdot
      \prod_{j\notin\mathbb{J}}  \lambda_j(y)^{-1/p}
      dy
      \cdot
      (M-2N)^{J/q+(d-J)/p}.
  \end{split}
  \end{equation}
\end{proof}

Hence, if the boundedness of $H_{\Phi_{\bbJ},\La_{\bbJ}} : \fJmpq(\bbR^d) \rightarrow \lqp(\bbR^J\times\bbR^{d-J})$ holds,
by the upper estimate (\ref{pp-FJM-Lqp-right}) and the lower estimate (\ref{pp-FJM-Lqp-left}),
we obtain that
\be
\begin{split}
	&\|H_{\Phi_{\bbJ},\La_{\bbJ}} \|_{\fJmpq(\bbR^d) \rightarrow \lqp(\bbR^J\times\bbR^{d-J})}
	\\
	\gtrsim &
	\int_{E_N}
	\Phi(y)
	\cdot \prod_{j\in\mathbb{J}} \lambda_j(y)^{1/q-1}
	\cdot
	\prod_{j\notin\mathbb{J}}  \lambda_j(y)^{-1/p}
	dy
	\cdot
	\frac{(M-2N)^{J/q+(d-J)/p}}{(M+2)^{J/q+(d-J)/p}}.
\end{split}
\ee
Letting $M\rightarrow\infty$, we have
\begin{align*}
  \int_{E_N}
      \Phi(y)
      \cdot \prod_{j\in\mathbb{J}} \lambda_j(y)^{1/q-1}
      \cdot
      \prod_{j\notin\mathbb{J}}  \lambda_j(y)^{-1/p}
      dy
  \lesssim
   \|H_{\Phi_{\bbJ},\La_{\bbJ}} \|_{\fJmpq(\bbR^d) \rightarrow \lqp(\bbR^J\times\bbR^{d-J})}.
\end{align*}
Finally, we obtain (\ref{pp-FJM-Lqp-condition}) by letting $N\rightarrow\infty$.
\begin{proof}[Proof of Theorem \ref{thm-bdh}]
We first check that the condition \eqref{thm-bdh-cd} implies the basic condition \eqref{weakest cond.}.
This follows by the fact
\be
1\wedge \la^{-1}\leq \max\{\la^{-1/p}, \la^{1/q-1}, \la^{-2/p+1/q}\}=\G_{p,q}(\la),\ \ \ \la>0.
\ee
By this fact and Proposition \ref{thm-dfH}, the Hausdorff operator $H_{\Phi,A}$ is well defined.
For $f, \va \in\calS(\rd)$, we write
\be
\begin{split}
V_{\va}(H_{\Phi,A}f)(x,\xi)
= &
\lan H_{\Phi,A}f, M_{\xi}T_x\va\ran
\\
= &
\int_{\rn}\Phi(y)\lan D_{A(y)}f(x),M_{\xi}T_x\va\ran dy
\\
= &
\int_{\rn}\Phi(y)V_{\va}(D_{A(y)}f)(x,\xi) dy.
\end{split}
\ee
Using Minkowski's inequality and
the dilation property of modulation space
 (Proposition \ref{pp-diagonal}), we conclude that
  \begin{align*}
    \|H_{\Phi, A}f\|_{\mpqd}
    &=
    \|V_{\va}(H_{\Phi,A}f)(x,\xi)\|_{L^{p,q}(\rd\times \rd)}
    \\
    &=
    \Big\|\int_{\rn}\Phi(y)V_{\va}(D_{A(y)}f)(x,\xi) dy
    \Big\|_{L^{p,q}(\rd\times \rd)}
    \\
    &\leq
      \int_{\rn}|\Phi(y)|\cdot
       \big\|
         V_{\va}(D_{A(y)}f)(x,\xi)
       \big\|_{L^{p,q}(\rd\times \rd)}
      dy
    \\
    &=
    \int_{\rn}|\Phi(y)|\cdot
       \big\|
         D(A)f
       \big\|_{\mpqd}
      dy
    \\
    &\lesssim
      \int_{\rn}|\Phi(y)|\cdot
       \prod_{j=1}^d \G_{p,q}(\lambda_j(y)) dy
       \cdot
       \|f\|_{\mpqd}.
  \end{align*}
  The boundedness of $H_{\Phi, A}$ is proved.

  Next, we turn to the converse direction in the range of $(1/p-1/2)(1/q-1/p)\geqslant 0$.
  Recall that in this case, we assume $\Phi\geq 0$ and $A(y)=\La(y)=\text{diag}\{\lambda_1(y),\cdots,\lambda_d(y) \}$.
  If the boundedness \eqref{thm-bdh-bd} holds, the Hausdorff operator is a continuous map from $\calS(\rd)$ into $\calS'(\rd)$.
  Then, Theorem \ref{thm-dfH} tells us that the basic condition \eqref{weakest cond.} holds.

  \textbf{Case 1: $1/2\leq 1/p \leq 1/q$.}
  It follows directly by the definition of $\G_{p,q}$ that for all $y\in \rn$,
  \[
    \prod_{j=1}^d \G_{p,q}(\lambda_j(y))
    \sim
    \sum_{\mathbb{J}\subset \{1,2,\cdots,d\}}
     \prod_{j\in\mathbb{J}} \lambda_j(y)^{1/q-1} \cdot \prod_{j\notin\mathbb{J}}  \lambda_j(y)^{-1/p},
  \]
  where the implicit constants are independent of $y$.
  Thus, the condition \eqref{thm-bdh-cd} is equal to
  \begin{align}
  \int_{\rn}
      \Phi(y)
      \cdot \prod_{j\in\mathbb{J}} \lambda_j(y)^{1/q-1}
      \cdot
      \prod_{j\notin\mathbb{J}}  \lambda_j(y)^{-1/p}
      dy
  <\infty, \ \ \ \ \ \text{for all }\mathbb{J}\subset\{1,2,\cdots,d\}.\label{thm-H-bdd-cond.2}
  \end{align}
  With the basic condition \eqref{weakest cond.}, the partial Fourier transform can be well defined on $H_{\Phi, A}f$ for $f\in \calS(\rd)$. 
  In fact, for any $f \in \calS(\rd)$ and $\bbJ\subset \{1,2,\cdots,d\}$,  we have $\fJ H_{\Phi,\La}f=H_{\Phi_{\bbJ},\La_{\bbJ}}\fJ f$.
  Using this and the boundedness of $H_{\Phi, A}$, we conclude that for any $f \in \calS(\rd)$ and $\bbJ\subset \{1,2,\cdots,d\}$,
  \begin{equation*}
    \begin{split}
      \| H_{\Phi_{\bbJ},\La_{\bbJ}}  \fJ f \|_{\fJmpqd}
      &=
      \|\fJ H_{\Phi,A}f\|_{\fJmpqd}
      =
      \|H_{\Phi,A}f\|_{\mpqd}
      \\
      &\lesssim
      \|f\|_{\mpqd}
      =
      \|\fJ f\|_{\fJmpqd}.
    \end{split}
  \end{equation*}
  So we obtain the boundedness of $H_{\Phi_{\bbJ},\La_{\bbJ}}$ on $\fJmpqd$ for all $\bbJ\subset \{1,2,\cdots,d\}$.
  Using this and Prposition \ref{pp-embed-fjmpq-lpq}, we obtain the boundedness
  $H_{\Phi_{\bbJ},\La_{\bbJ}}: \fJmpq(\bbR^d) \rightarrow \lqp(\bbR^J\times\bbR^{d-J})$ for all $\bbJ\subset \{1,2,\cdots,d\}$.
  Then, (\ref{thm-H-bdd-cond.2}) follows by this and Proposition \ref{pp-FJM-Lqp}.

  \textbf{Case 2: $1/q \leq 1/p \leq 1/2 $.}
  Using Proposition \ref{R. Adj. Hausdorff} and the dual property of modulation space (see \cite[Lemma 2.2]{Benyi2005} for instance), one can find that
  $H^*_{\Phi,A}$ is bounded on $M^{p',q'}$ if $H_{\Phi, A}$ is bounded on $M^{p,q}$.
  Observing  that $1/2\leq 1/p' \leq 1/q'$, 
  and recaling $H^*_{\Phi,A}=H_{\wt{\Phi}, A^{-1}}$ proved in Proposition \ref{R. Adj. Hausdorff}, 
  by the conclusion proved in Case 1, we obtain that
  \begin{align}
    \int_{\rn}\wt{\Phi}(y) \cdot
    \prod_{j=1}^d \G_{p',q'}(\lambda^{-1}_j(y)) dy
    <
    \infty.
  \end{align}
  Observe that
   \begin{align}
     \wt{\Phi}(y) \cdot \prod_{j=1}^d \G_{p',q'}(\lambda^{-1}_j(y))
      =
      \Phi(y)\cdot \prod_{j=1}^d \G_{p,q}(\lambda_j(y)).
   \end{align}
   The above two estimates yiled the desired conclusion
   \begin{align*}
   \int_{\rn}\Phi(y)\cdot\prod_{j=1}^d \G_{p,q}(\lambda_j(y)) dy
   <\infty.
   \end{align*}
   We have now completed the proof.
\end{proof}
\subsection*{Acknowledgements}
This work was partially supported by the
Natural Science Foundation of Fujian
Province (Nos.2020J01267, 2020J01708, 2021J011192)

\bibliographystyle{abbrv}

\begin{thebibliography}{10}
	
	\bibitem{BrownMoricz2002JMAA}
	G.~Brown and F.~M\'{o}ricz.
	\newblock Multivariate {H}ausdorff operators on the spaces {$L^p(\Bbb R^n)$}.
	\newblock {\em J. Math. Anal. Appl.}, 271(2):443--454, 2002.
	
	\bibitem{Benyi2005}
	A.~Bényi, K.~Gröchenig, C.~Heil, and K.~Okoudjou.
	\newblock Modulation spaces and a class of bounded multilinear
	pseudodifferential operators.
	\newblock {\em Journal of Operator Theory}, 54(2):389--401, 2005.
	
	\bibitem{ChenFanLi2012CAMSB}
	J.~Chen, D.~Fan, and J.~Li.
	\newblock Hausdorff operators on function spaces.
	\newblock {\em Chin. Ann. Math. Ser. B}, 33(4):537--556, 2012.
	
	\bibitem{Chen2013}
	J.~Chen, D.~Fan, and S.~Wang.
	\newblock Hausdorff operators on {E}uclidean spaces.
	\newblock {\em Appl. Math. J. Chinese Univ. Ser. B}, 28(4):548--564, 2013.
	
	\bibitem{ChenFanZhang2012AMSES}
	J.~C. Chen, D.~S. Fan, and C.~J. Zhang.
	\newblock Multilinear {H}ausdorff operators and their best constants.
	\newblock {\em Acta Math. Sin. (Engl. Ser.)}, 28(8):1521--1530, 2012.
	
	\bibitem{CorderoNicola2007JoFA}
	E.~Cordero and F.~Nicola.
	\newblock Metaplectic representation on {W}iener amalgam spaces and
	applications to the {S}ch\"{o}rdinger equation.
	\newblock {\em Journal of Functional Analysis}, 254(2):506--534, 2007.
	
	\bibitem{Cordero2009}
	E.~Cordero and F.~Nicola.
	\newblock Sharpness of some properties of {W}iener amalgam and modulation
	spaces.
	\newblock {\em Bulletin of the Australian Mathematical Society},
	80(1):105--116, 2009.
	
	\bibitem{FanLin2014AB}
	D.~Fan and X.~Lin.
	\newblock Hausdorff operator on real {H}ardy spaces.
	\newblock {\em Analysis (Berlin)}, 34(4):319--337, 2014.
	
	\bibitem{Feichtinger1983TRUoV}
	H.~G. Feichtinger.
	\newblock Modulation spaces on locally compact abelian groups.
	\newblock {\em Technical Report, University of Vienna}, 1983.
	
	\bibitem{GrochenigBook2013}
	K.~Gr{\"{o}}chenig.
	\newblock {\em {Foundations of Time-Frequency Analysis}}.
	\newblock Springer Science {\&} Business Media, 2013.
	
	\bibitem{KarapetyantsLiflyand2020MMitAe}
	A.~Karapetyants and E.~Liflyand.
	\newblock Defining {H}ausdorff operators on {E}uclidean spaces.
	\newblock {\em Mathematical Methods in the Applied ences}, (4), 2020.
	
	\bibitem{LernerLiflyand2007JAMS}
	A.~K. Lerner and E.~Liflyand.
	\newblock Multidimensional {H}ausdorff operators on the real {H}ardy space.
	\newblock {\em J. Aust. Math. Soc.}, 83(1):79--86, 2007.
	
	\bibitem{Liflyand2011}
	E.~Liflyand.
	\newblock {Complex and real Hausdorff operators}.
	\newblock {\em CRM, preprint 1046, 2011, 45p. Or http://www.
		crm.es/Publications/11/Pr1046.pdf}.
	
	\bibitem{Miyachi2004JFAA}
	A.~Miyachi.
	\newblock Boundedness of the {C}es\`aro operator in {H}ardy spaces.
	\newblock {\em J. Fourier Anal. Appl.}, 10(1):83--92, 2004.
	
	\bibitem{Siskakis1990PAMS}
	A.~G. Siskakis.
	\newblock The {C}es\`aro operator is bounded on {$H^1$}.
	\newblock {\em Proc. Amer. Math. Soc.}, 110(2):461--462, 1990.
	
	\bibitem{SugimotoTomita2007JoFA}
	M.~Sugimoto and N.~Tomita.
	\newblock The dilation property of modulation spaces and their inclusion
	relation with {B}esov spaces.
	\newblock {\em Journal of Functional Analysis}, 248(1):79--106, 2007.
	
	\bibitem{Triebel1983ZFA}
	H.~Triebel.
	\newblock {Modulation Spaces on the Euclidean n-Space}.
	\newblock {\em Zeitschrift f{\"{u}}r Analysis und ihre Anwendungen},
	2(5):443--457, 1983.
	
	\bibitem{Wang2011Book}
	B.~Wang, Z.~Huo, C.~Hao, and Z.~Guo.
	\newblock {\em {Harmonic Analysis Method for Nonlinear Evolution Equations,
			I}}.
	\newblock World Scientific, 2011.
	
	\bibitem{ZhaoFanGuo2018AFA}
	G.~Zhao, D.~Fan, and W.~Guo.
	\newblock Hausdorff operators on modulation and {W}iener amalgam spaces.
	\newblock {\em Ann. Funct. Anal.}, 9(3):398--412, 2018.
	
\end{thebibliography}

\end{document}